\documentclass[12pt]{amsart}
\usepackage{url, mathrsfs}
\usepackage{hyperref}
\usepackage{amsmath}
\usepackage{graphicx, 
epstopdf}
\usepackage{subfig}
\usepackage{color}
\usepackage{bbm}
\usepackage[draft]{fixme}
\usepackage{bm} 
\usepackage{bbm} 
\usepackage{epigraph}
\usepackage{hyperref}
\usepackage{endnotes}
\usepackage{ifpdf}
\usepackage{array}
\usepackage{multicol}
\usepackage{multirow}
\usepackage{graphicx}
\usepackage{tikz}

\newcommand{\N}{\mathbb{N}}
\newcommand{\Z}{\mathbb{Z}}
\newcommand{\R}{\mathbb{R}}
\newcommand{\F}{\mathbb{F}}
\newcommand{\C}{\mathbb{C}}
\renewcommand{\P}{\mathbb{P}}

\newcommand{\V}{\mathcal{V}}
\newcommand{\HH}{\mathcal{H}}
\newcommand{\OO}{\mathcal{O}}
\newcommand{\W}{\mathcal{W}}

\newcommand{\x}{\mathbf{x}}

\newcommand{\floor}[1]{\left\lfloor #1 \right\rfloor}
\newcommand{\es}[1]{\Lambda(\omega^{#1})}
\newcommand{\esl}{\Lambda(\omega^{\ell})}
\newcommand{\adj}[1]{\text{adj}(#1)}
\newcommand{\Ts}[1]{T_s^{#1}}
\newcommand{\CWS}{\mathcal{C}}
\newcommand{\conj}{{\rm conj}}
\newcommand{\refl}{{\rm ref}}
\newcommand{\rot}{{\rm rot}}
\newcommand{\hi}{\textbf}

\theoremstyle{plain}
\newtheorem{thm}{Theorem}[section]
\newtheorem{prop}[thm]{Proposition}
\newtheorem{cor}[thm]{Corollary}
\newtheorem{lem}[thm]{Lemma}
\newtheorem{question}[thm]{Question}
\newtheorem{conjecture}[thm]{Conjecture}

\newtheorem*{claim*}{Claim}
\newtheorem*{thm*}{Theorem} 

\newtheoremstyle{case}{}{}{}{}{}{:}{ }{}
\theoremstyle{case}

\theoremstyle{definition}
\newtheorem{defn}[thm]{Definition}
\newtheorem{example}[thm]{Example}
\newtheorem{remark}[thm]{Remark}
\newtheorem{construction}[thm]{Construction}

\DeclareMathOperator{\GL}{GL}
\DeclareMathOperator{\diag}{diag}

\usepackage[inner=1in,outer=1in,top=1in,bottom=1in]{geometry}

\begin{document}
\title[Invariant hyperbolic plane curves]{Invariant hyperbolic curves:\\ 
determinantal representations and \\applications to the numerical range}

\author{Faye Pasley Simon}
\address{Greensboro College, Greensboro, NC 27401} 
\email{faye.simon@greensboro.edu}

 \author{Cynthia Vinzant}
\address{North Carolina State University, Raleigh, NC, USA 27695}
\email{clvinzan@ncsu.edu}

\begin{abstract}
Here we study the space of real hyperbolic plane curves that are invariant under actions of the cyclic and dihedral groups
and show they have determinantal representations that certify this invariance.  
We show an analogue of Nuij's theorem for the set of invariant hyperbolic polynomials of a given degree. 
The main theorem is that every invariant hyperbolic plane curve has a determinantal representation 
using a block cyclic weighted shift matrix. This generalizes previous work by Lentzos and the first author, as well as by Chien and Nakazato.
One consequence is that if the numerical range of a matrix is invariant under rotation, then 
it is the numerical range of a block cyclic weighted shift matrix. 
\end{abstract}



\maketitle

\section{Introduction}\label{sec:intro}

Here we study properties of a real plane curve that can be certified by a 
Hermitian determinantal representation, in particular, hyperbolicity and invariance under the action of a finite group. 
A real homogeneous polynomial is \emph{hyperbolic} with respect to a point in $\R^n$ if it is positive at the point and has real-rooted restrictions on every line through that point. 

Hyperbolic polynomials were introduced  in the mid-20th century by 
Petrovsky and G{\aa}rding, in the context of partial differential equations. 
Since then they have appeared in a wide range of areas and applications, 
including convex optimization \cite{Guler97, Renegar06}, 
combinatorics \cite{LiggetNegDep, MonCol, GurvitsPerm, RamGraph}, convex and complex analysis \cite{ConvAnalysis, BB_StabPres}, and operator theory \cite{HeltonVinnikov07,kadSing}.

A fundamental example is given by the determinant. 
On the real vector space of Hermitian matrices, the 
determinant is hyperbolic with respect to the 
identity matrix. 
More generally, given a linear matrix pencil $\mathcal{A}({\bf x}) = \sum_{i=1}^nx_iA_i$ where 
the matrices $A_1, \hdots, A_n$ are Hermitian and the matrix $\mathcal{A}({\bf e})$ is positive definite, 
the polynomial $f({\bf x}) = \det(\mathcal{A}({\bf x}))$ is hyperbolic with respect to ${\bf e}\in \R^n$. 
This determinantal representation certifies the hyperbolicity of $f$ and we say that $\mathcal{A}({\bf x})$ 
is a definite determinantal representation of $f$. 
See \cite{VinnikovSurvey} for more. 

\begin{figure}
\includegraphics[height=1.5in]{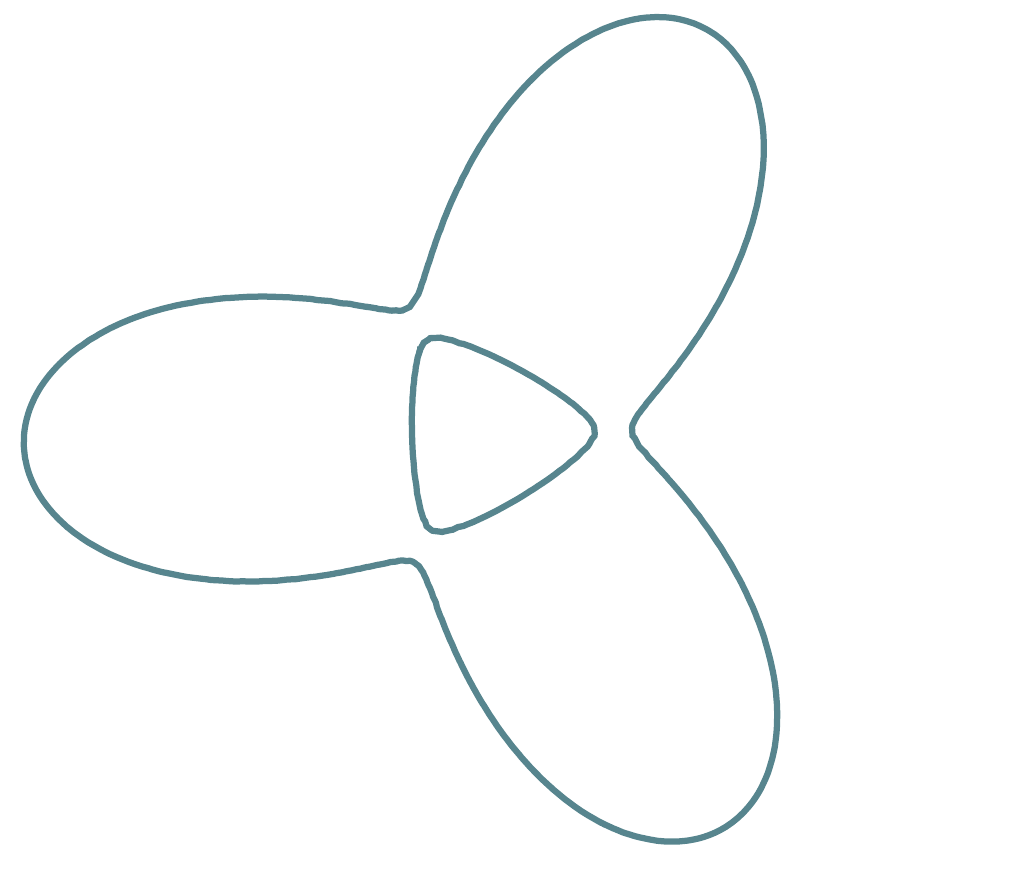} \ \ \ \ \ \ \ \includegraphics[height=1.5in]{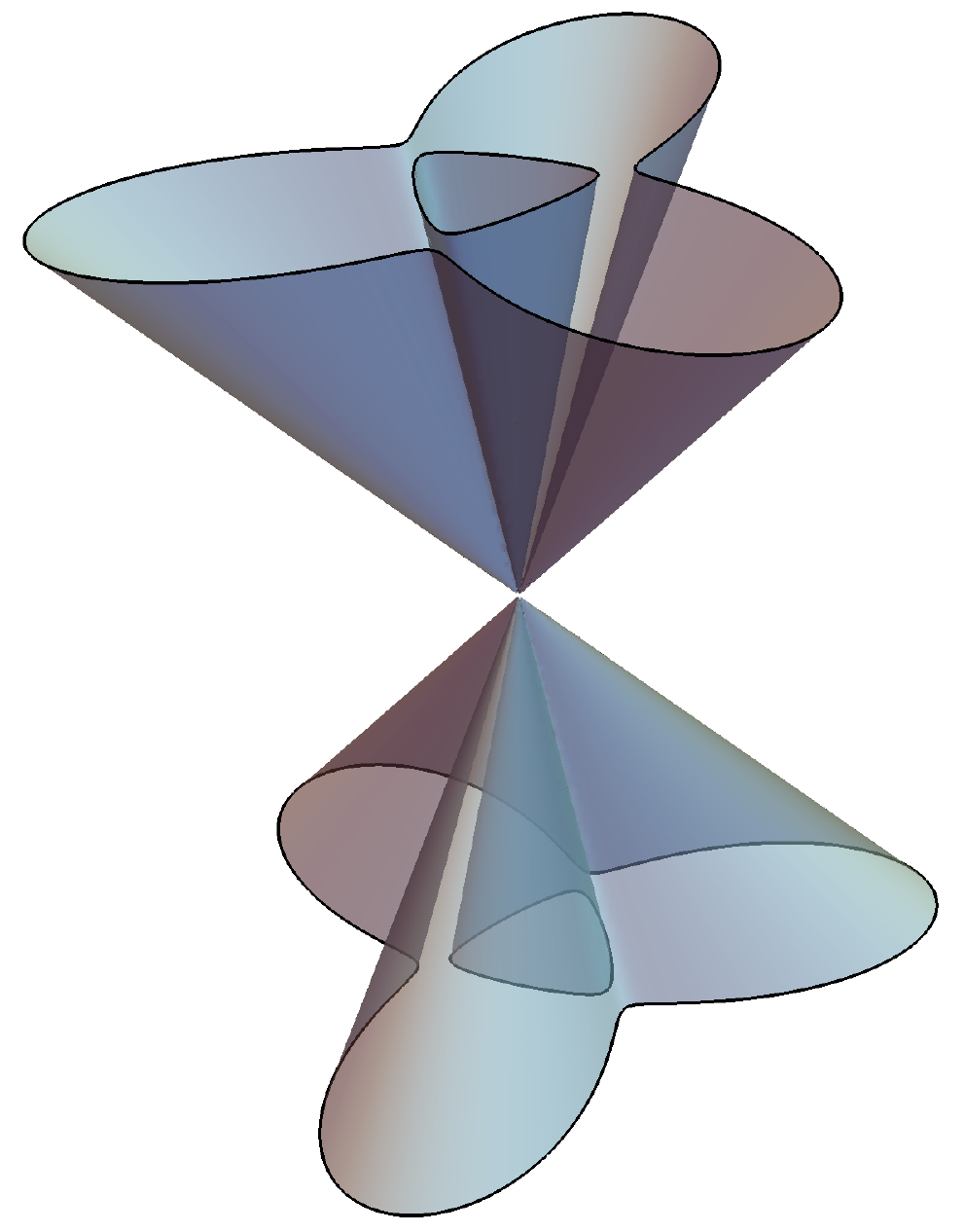}
\caption{The variety of a quartic invariant hyperbolic form in $\P^2(\R)$ and $\R^3$.}
\label{fig:quartic}
\end{figure}

For $n=3$, $\V_{\C}(f)$ is a plane curve in $\mathbb{P}^2(\C)$. 
Determinantal representations are a classical object of study \cite{Beauville00, Dolgachev}. 
In 1902, Dixon showed that every plane curve has a symmetric determinantal 
representation over the complex numbers \cite{Dixon}. 
Almost a hundred years later, Helton and Vinnikov proved the Lax conjecture, showing 
that every hyperbolic plane curve has a definite determinantal representation 
with real symmetric matrices \cite{HeltonVinnikov07}. 
Proving the existence of such representations involves the existence 
of two-torsion points on the Jacobian of the curve with certain real structure. 
Showing the existence of definite Hermitian representations is less delicate. 
Concrete methods for constructing  such representations were studied  by Plaumann and Vinzant \cite{PV}.

Here we study this question in the context of curves invariant under the action 
of a finite group, in particular the cyclic or dihedral groups. 
We say that a polynomial $f\in \R[t,x,y]$ is invariant under 
a group $\Gamma \subset \GL(\R^3)$ if $f(\gamma\cdot(t,x,y)) = f$ 
for all $\gamma\in \Gamma$ . 
We will be interested in the cyclic and dihedral groups 
$C_n = \langle {\rm rot}\rangle$ and $D_{2n} = \langle {\rm rot}, {\rm ref}\rangle$ on $\R^3$,
given by 
\begin{equation}\label{eq:CnDnDef}
{\rm rot}\cdot \begin{pmatrix}t\\x\\y\end{pmatrix} = 
\begin{pmatrix}
1& 0 & 0\\ 0 & \cos(2\pi/n) & \sin(2\pi/n)\\ 0 &   -\sin(2\pi/n) & \cos(2\pi/n)
\end{pmatrix} \begin{pmatrix}t\\x\\y\end{pmatrix}
\ \ \text{ and } \ \
{\rm ref}\cdot \begin{pmatrix}t\\x\\y\end{pmatrix} = \begin{pmatrix}\ t\\ \ x\\-y\end{pmatrix}.
\end{equation}

Our first main theorem is an analogue of Nuij's theorem on the structure of 
the set of hyperbolic polynomials of a given degree.

\newtheorem*{thm1}{Theorem~\ref{thm:InvHyp}}
\begin{thm1}
The set of polynomials in $\R[t,x,y]_d$  that are hyperbolic with respect to 
$(1,0,0)$ and invariant under the action of the cyclic or dihedral group (of any order) 
is contractible and equal to the closure of its interior in the Euclidean topology on $\R[t,x,y]_d$.
\end{thm1}

This is a key step in the proof that all such polynomials have an invariant definite
determinantal representation. 
Such representations were first studied by Chien and Nakazato in the context of 
numerical ranges \cite{Chien2013}. 
Given a matrix $A\in \C^{d \times d}$, define the polynomial 
\begin{equation}\label{eq:FA}
F_A(t,x,y)  \ = \ \det(t I + x (A+A^*)/2 + y (A-A^*)/2i) \ \in  \ \R[t,x,y]_d.
\end{equation}
Since the matrices $I$, $(A+A^*)/2$ and $(A-A^*)/2i$ are Hermitian and the identity matrix 
is positive definite, $F_A$ has a definite determinantal representation and is hyperbolic 
with respect to $(1,0,0)$. 
Chien and Nakazato show that if $A\in \C^{n\times n}$ 
is a complex \emph{cyclic weighted shift matrix}, 
then $F_A$ is invariant under the action of the cyclic group $C_n$ and if additionally 
$A$ has real entries then $F_A$ is invariant under the action of the dihedral group $D_{2n}$ \cite{Chien2013}. 
They also show that for $n=3$ and $4$ any hyperbolic, invariant polynomial $f\in \R[t,x,y]_n$ has such a 
representation. This was generalized by Lentzos and Pasley \cite{LP} who show this for all $n$. 

Here we generalize this to \emph{block cyclic weighted shift matrices}, as defined in 
Definition~\ref{defn: CWS}. For such a matrix $A$, the polynomial $F_A$ 
is hyperbolic with respect to $(1,0,0)$ and invariant under $C_n$ or $D_{2n}$, if 
additionally the matrix is real.  Moreover, any invariant hyperbolic polynomial has such a representation when its degree is an integer multiple of $n$.

\newtheorem*{thm2}{Theorem~\ref{thm:d=qn}}
 \begin{thm2}
	Let $d\in n\Z_{+}$ and suppose $f\in \R[t,x,y]_{d}$ is hyperbolic 
	with respect to $(1,0,0)$, with $f(1,0,0)=1$, and invariant under the action of ${\Gamma}$.
	\begin{itemize}
		\item[(a)] If $\Gamma = C_n$, then $f=F_A$ for some block cyclic weighted shift matrix $A \in \C^{d \times d}$.
		\item[(b)] If $\Gamma = D_{2n}$,  then $f=F_A$ for some block cyclic weighted shift matrix $A \in \R^{d \times d}$.
	\end{itemize}
\end{thm2}

The original motivation of Chien and Nakazato was to understand invariance of numerical ranges.
Formally, the numerical range of a matrix $A\in \C^{d\times d}$ is 
\[
\W(A) \ = \ \left\{ {\bf v}^*A{\bf v}  \ : \ {\bf v}\in \C^d, \ ||{\bf v}||=1 \right\} \ \subset \ \C.
\]
The Toeplitz-Hausdorff theorem states that this is a convex body in $\C\cong \R^2$ \cite{Hausdorff, Toeplitz}.
This set appears in applications related to engineering, numerical analysis, and differential equations~\cite{axelsson, cheng, eiermann, fiedler, goldberg}.

\begin{thm*}[Kippenhahn \cite{Kippenhahn}]
 Let $X^*$ be the dual variety to $X = V_{\C}(F_A)$. 
 The numerical range of $A\in \C^{d \times d}$ is the convex hull of the real, affine part of $X^*$. That is,
\[
\W(A) \ = \ {\rm conv}\left(\{x+iy : [1:x:y]\in X^*(\R)\}\right).
\]
\end{thm*}

Using the above theorem on invariant determinantal representations, 
we show the following about numerical ranges that are invariant under the action of the cyclic or dihedral group.

\newtheorem*{thm3}{Theorem~\ref{thm: num range}}
\begin{thm3}
Let $A\in \C^{d\times d}$ and let $\W(A)$ denote its numerical range. 
If $\W(A)$ is invariant under multiplication by $n$-th roots of unity, then there exists a block cyclic weighted shift matrix $B$ 
of size $\leq n\cdot \lceil d/n\rceil$ so that $\W(A) = \W(B)$. Moreover if $\W(A)$ is invariant under conjugation, then the entries of $B$ can be taken in $\R$. 
\end{thm3}

\begin{figure}
\includegraphics[height=1.5in]{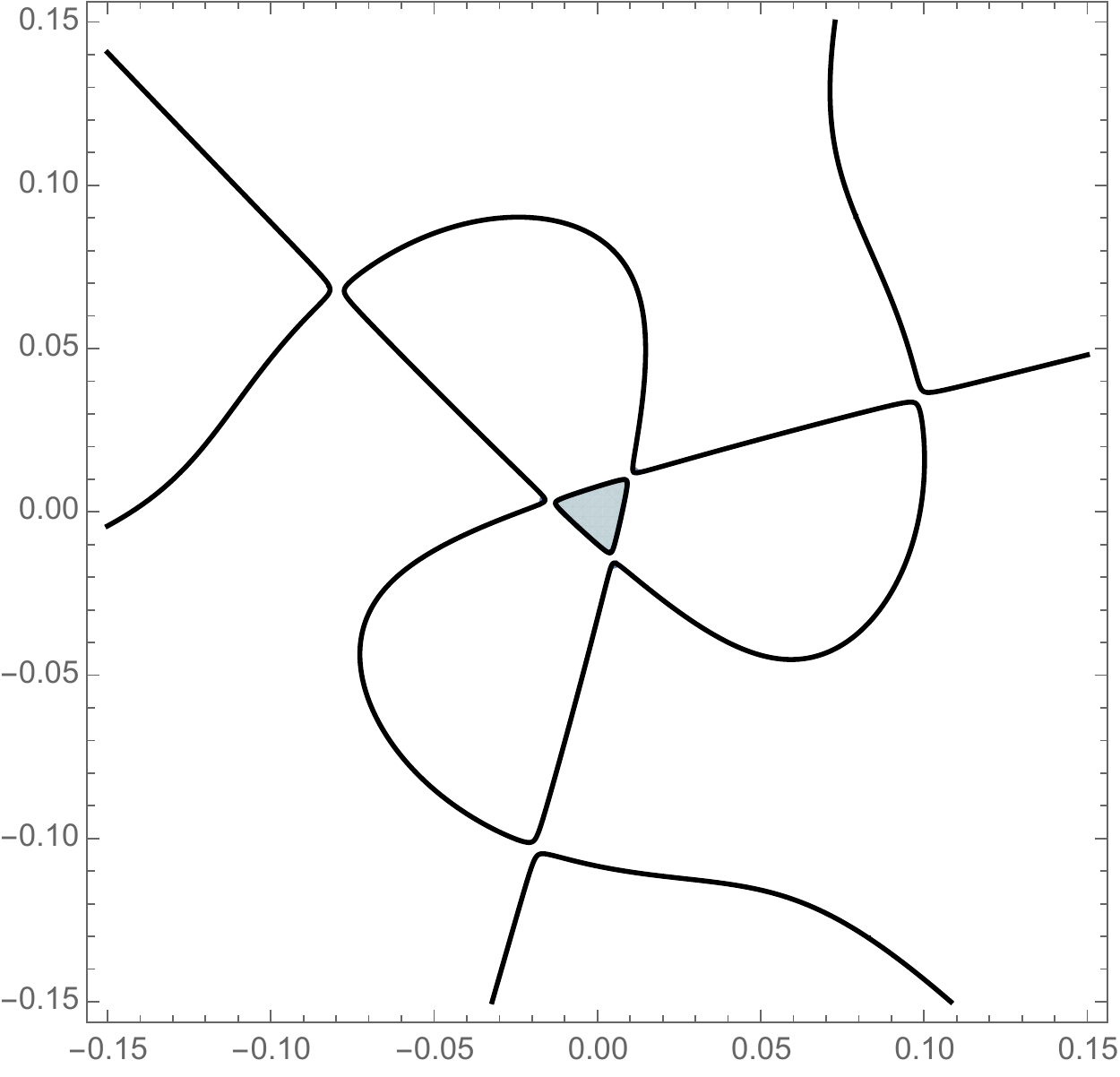} \ \ \ \ \ \ \ \includegraphics[height=1.5in]{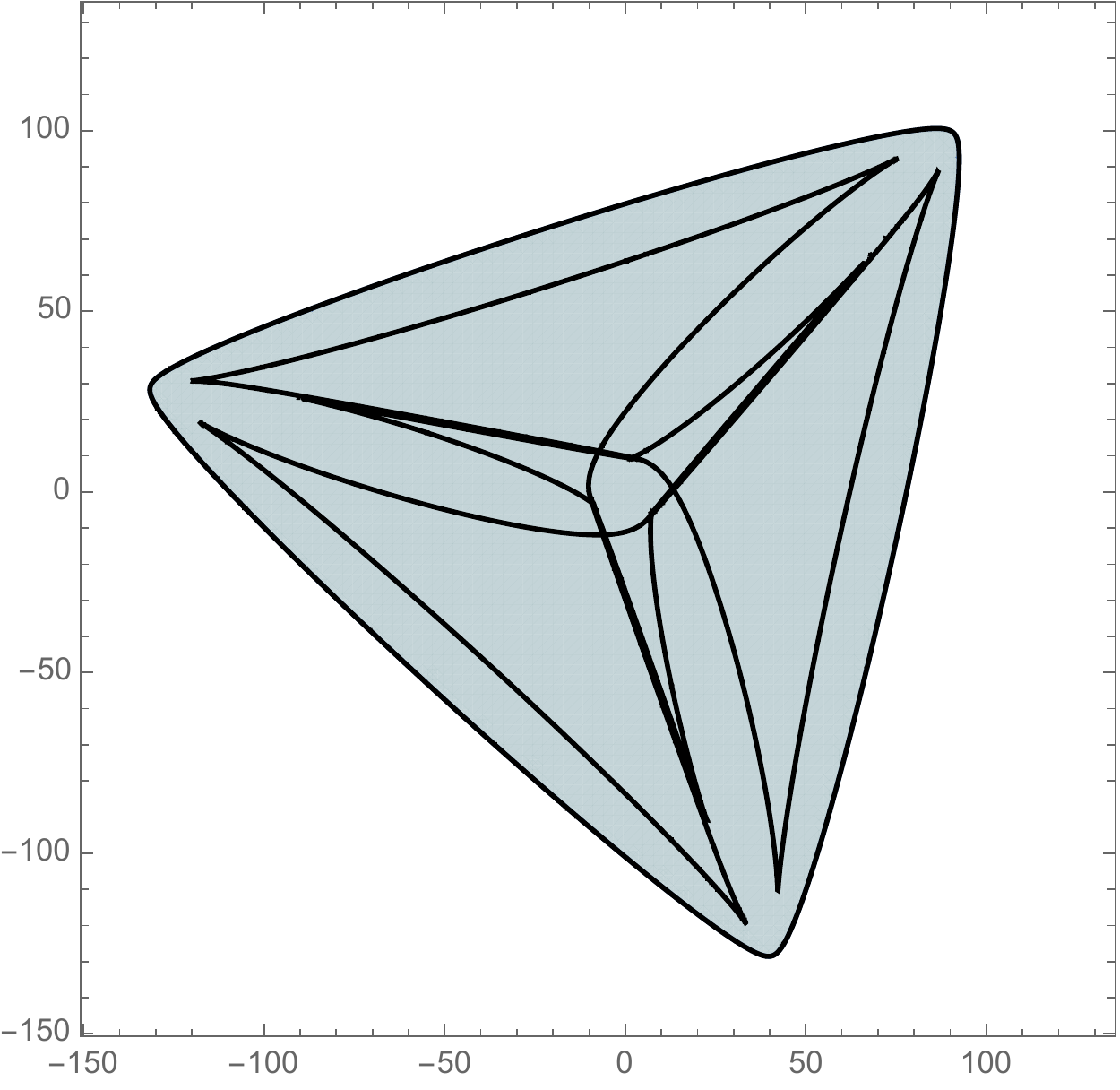}
\caption{The curve $\V_{\R}(F_A)$ and its dual curve  bounding $\W(A)$ for $A\in \C^{5\times 5}$.}
\label{fig:numRange}
\end{figure}

The paper is organized as follows. In Section~\ref{sec:Nuij}, we introduce the theory of hyperbolic polynomials and prove an invariant analogue of Nuij's theorem on the topology of this set.  The precise definition of block cyclic weighted shift matrices and their connection to invariant hyperbolic polynomials is discussed in Section~\ref{sec:CWS}. 
 In Sections~\ref{sec: d=qn smooth}~and~\ref{sec: d=qn dihedral}, we prove parts (a) and (b) of Theorem~\ref{thm:d=qn} under some genericity conditions on the curve $\V_{\C}(f)$ and in Section~\ref{sec: topology} we address the degenerate cases to complete the proof.  
Applications to numerical ranges are given in Section~\ref{sec: num range}. 
Finally we conclude with a discussion of open problems in Section~\ref{sec:questions}.\\

\noindent {\bf Acknowledgements.}  We thank Ricky Liu, Hiroshi Nakazato, Linda Patton, Daniel Plaumann, Edward Poon, and Rainer Sinn 
for helpful comments and discussions.
Part of this work was done while both authors were participants 
at the Fall 2018 Nonlinear Algebra program at the Institute for Computational and Experimental Research in Mathematics.
Both authors were partially supported by the US NSF-DMS grant \#1620014. 
The second author was also partially supported by the US NSF-DMS grant \#1943363. 
This material is based upon work directly supported by the National Science Foundation Grant No. DMS-1926686, and indirectly supported by the National Science Foundation Grant No. CCF-1900460.

 \section{Invariant hyperbolic polynomials} \label{sec:Nuij} 
 For a field $\mathbb{F} = \R$ or $\C$, we use 
 $\mathbb{F}[t,x,y]_d$ to denote the $\mathbb{F}$-vectorspace of polynomials in variables $t,x,y$ 
 that are homogeneous of degree $d$. 
 Given $f\in \mathbb{F}[t,x,y]_d$, we use $\V_{\F}(f)$ to denote the variety of $f$ in the projective plane 
 $\P^2(\F)$ over $\F$. 
 A finite group $\Gamma$ of $ {\rm GL}(\R^3)$ defines an action on the vector space $\mathbb{F}[t,x,y]$
given by $\gamma\cdot f = f(\gamma\cdot(t,x,y))$. 
Let $\mathbb{F}[t,x,y]^{\Gamma}$ denote the subring of invariant polynomials $f$ satisfying $f(\gamma\cdot(t,x,y)) = f$ for all $\gamma\in \Gamma$ and let $\mathbb{F}[t,x,y]^{\Gamma}_d$
denote homogeneous elements of degree $d$ in this ring. 
By a classical theorem of Hilbert, the invariant ring  $\mathbb{F}[t,x,y]^{\Gamma}$ is finitely generated. 
 For the group actions of $C_n$ and $D_{2n}$ given in \eqref{eq:CnDnDef}, we can explicitly find these generators.
 See e.g.~\cite{GTW2013}. Namely, 
 \begin{align*}
\R[t,x,y]^{C_n} &\ = \ \R\!\left[t, x^2+y^2, \mathfrak{R}[(x+iy)^n] , \mathfrak{I}[(x+iy)^n]\right],  \  \text{ and }  \\
\R[t,x,y]^{D_{2n}} &\ = \ \R\!\left[t, x^2+y^2, \mathfrak{R}[(x+iy)^n]\right]
 \end{align*}  
 where 
\begin{equation}
 \mathfrak{R}[(x+iy)^n] = \frac{(x+iy)^n + (x-iy)^n}{2} \ \ \text{ and } \ \  \mathfrak{I}[(x+iy)^n] = \frac{(x+iy)^n-(x-iy)^n}{2i}.
\end{equation}

 We will be particularly interested in the set of invariant \emph{hyperbolic} polynomials.
 
\begin{defn}A polynomial $f\in \R[t,x,y]_d$ is {\bf hyperbolic} with respect to a point ${\bf e}\in \R^3$ if 
$f({\bf e})> 0$ and $f(\lambda {\bf e} - {\bf v})\in \R[\lambda]$ 
is real-rooted for every choice of ${\bf v}\in \R^3$.  We call $f$ {\bf strictly hyperbolic} with respect to 
${\bf e}$ if $f$ is hyperbolic with respect to ${\bf e}$ and the roots of $f(\lambda {\bf e} - {\bf v})$ are 
distinct for every ${\bf v} \in \R^3\backslash (\R{\bf e})$. 
By \cite[Lemma~2.4]{PV}, an equivalent definition is that $f$ is hyperbolic with respect to ${\bf e}$
and its real projective variety $\V_{\R}(f)$ is smooth. 

A polynomial $g\in \R[t,x,y]_{d-1}$ is {\bf interlaces} $f$ with respect to ${\bf e}\in \R^3$
if both are hyperbolic with respect to ${\bf e}$ and the roots of $g(\lambda {\bf e} - {\bf v})$ interlace 
the roots of $f(\lambda {\bf e} - {\bf v})$ for every ${\bf v}\in \R^3$.
We say that $g$  {\bf strictly interlaces}  $f$ with respect to 
${\bf e}$ if  $g$ interlaces  $f$ with respect to ${\bf e}$ the roots 
of $g(\lambda {\bf e} - {\bf v})$ and $f(\lambda {\bf e} - {\bf v})$ are all distinct 
for every ${\bf v} \in \R^3\backslash (\R{\bf e})$. 
 \end{defn}

For $\Gamma = C_n, D_{2n}$ denote the set of hyperbolic, invariant forms of degree $d$ by 
\[
\mathcal{H}_d^{\Gamma} = \left\{ f\in \R[t,x,y]_d^{\Gamma} \ : \ f(1,0,0) = 1, \  f \text{ is hyperbolic with respect to }(1,0,0)\right\}. 
\]
The subset of hyperbolic polynomials without any real singularities we denote by 
\[
(\mathcal{H}^\circ)_d^{\Gamma}
= \left\{ f\in \mathcal{H}_d^{\Gamma} \  : \  \V_{\R}(f) \subset \P^2(\R)\text{ is smooth} \right\}.
\]
As noted above, these are exactly the strictly hyperbolic forms in $\mathcal{H}_d^{\Gamma}$. 
Polynomials in $(\mathcal{H}^\circ)_d^{\Gamma}$ may have complex singularities and indeed 
may be forced to do for some choices of $d$, as discussed in Section~\ref{subsec:singularities}.

Nuij \cite{nuij} showed that the set of hyperbolic polynomials of a given degree is contractible in 
$\R[t,x,y]_d\cong \R^{\binom{d+2}{2}}$ and equal to the closure of its interior, which consists of strictly 
hyperbolic polynomials. 
Here we show an analogous statement for hyperbolic polynomials invariant under the cyclic and dihedral groups.

\begin{thm} \label{thm:InvHyp}
For $\Gamma = C_n$ or $D_{2n}$ and any $d\in \Z_+$, 
both $(\mathcal{H}^\circ)_d^{\Gamma}$ and $\mathcal{H}_d^{\Gamma}$ are contractible. Moreover, 
$(\mathcal{H}^\circ)_d^{\Gamma}$ is 
a full-dimensional, open subset of the set of polynomials in $\R[t,x,y]_d^{\Gamma}$ with coefficient of $t^d$ equal to $1$ and 
its closure equals $ \mathcal{H}_d^{\Gamma} $.
\end{thm} 

The proof requires developing an invariant version of techniques used in \cite{nuij}. 
To understand how the sets $(\HH^\circ)_d^{\Gamma}$ and $\HH_d^{\Gamma}$ relate, we introduce the following linear operator on invariant polynomials. 
For $s \in \R$, define the linear map $\Ts{}: \R[t,x,y]_d \to \R[t,x,y]_d$ by 
\[
\Ts{}(f) = f - s^2 (x^2+y^2) \frac{\partial ^2 f}{\partial t^2}\text{.}
\]

\begin{lem}\label{lem:T_operator}
For any $s\in \R_{>0}$, the map $\Ts{}$ preserves invariance under $\Gamma$ and hyperbolicity. That is, 
$\Ts{}(\mathcal{H}_d^{\Gamma} ) \subset \mathcal{H}_d^{\Gamma}$. 
Moreover for any $f\in \mathcal{H}_d^{\Gamma}$, the polynomial $\Ts{d}(f)$, obtained by applying $\Ts{}$ $d$ times to $f$, 
is strictly hyperbolic with respect to $(1,0,0)$. That is $\Ts{d}(\mathcal{H}_d^{\Gamma} ) \subset (\mathcal{H}^\circ)_d^{\Gamma}$.
\end{lem}
	
\begin{proof} First, note that if $f\in \R[t,x,y]_d^{\Gamma}$, then so are $x^2+y^2$ and $ \frac{\partial ^2 f}{\partial t^2}$, meaning that $\Ts{}$ preserves invariance under $\Gamma$.  

For the other claims, consider the operator on univariate polynomials $T: \R[t] \to \R[t]$ where $T(p) = p-s^2p''$.
We claim that for any real-rooted polynomial $p\in \R[t]$, $T(p)$ is also real rooted and the 
roots of $T^{d}(p)$ where $d= \deg(p)$ are simple.
To see this, consider the maps $T_{\pm}: \R[t] \to \R[t]$ where $T_{\pm}(p) = p\pm sp'$ for some $s \in \R$. The roots of $T_{\pm}(p)$ have multiplicity one less than those of $p$, any repeated roots of $T_{\pm}(p)$ are also repeated roots of $p$ and any added roots of $T_{\pm}(p)$ are simple by the lemma of \cite{nuij}. Let $T = T_{+} \circ T_{-}$ so $T(p) = p-s^2p''$. The roots of $T(p)$ have multiplicity two less than those of $p$, and any repeated roots are also repeated roots of $p$. Any other roots of $T(p)$ are simple. If $d = \deg(p)$, this implies every root of $T^d(p)$ is simple.

Since for any $(a,b)\in \R^2$, the restriction $\Ts{}(f)(t,a,b)$ equals the image of $p(t) = f(t,a,b)$ under the univariate 
operator $T$, the polynomial $\Ts{}(f)$ is hyperbolic with respect to $(1,0,0)$ and $\Ts{d}(f)$ is strictly hyperbolic. 
\end{proof}	

For example, the left-most curve in Figure~\ref{fig:gallery} 
is defined by an element $f$ of $\mathcal{H}_6^{C_5}$ and has real singularities. Just to its right is the smooth curve defined by $T_s(f)$ in $(\mathcal{H}^\circ)_6^{C_5}$. 

\begin{figure}
\includegraphics[height=1.5in]{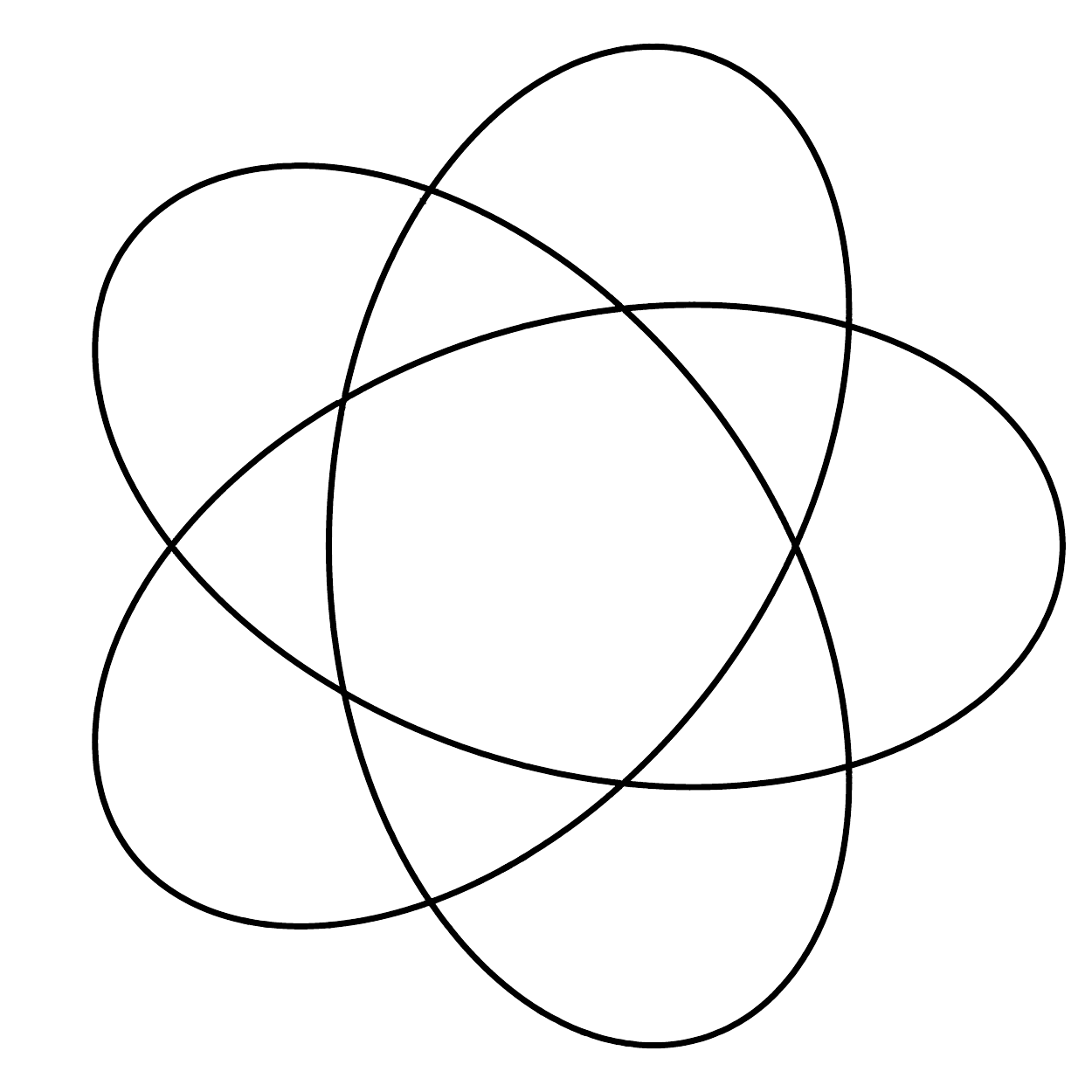}
\includegraphics[height=1.5in]{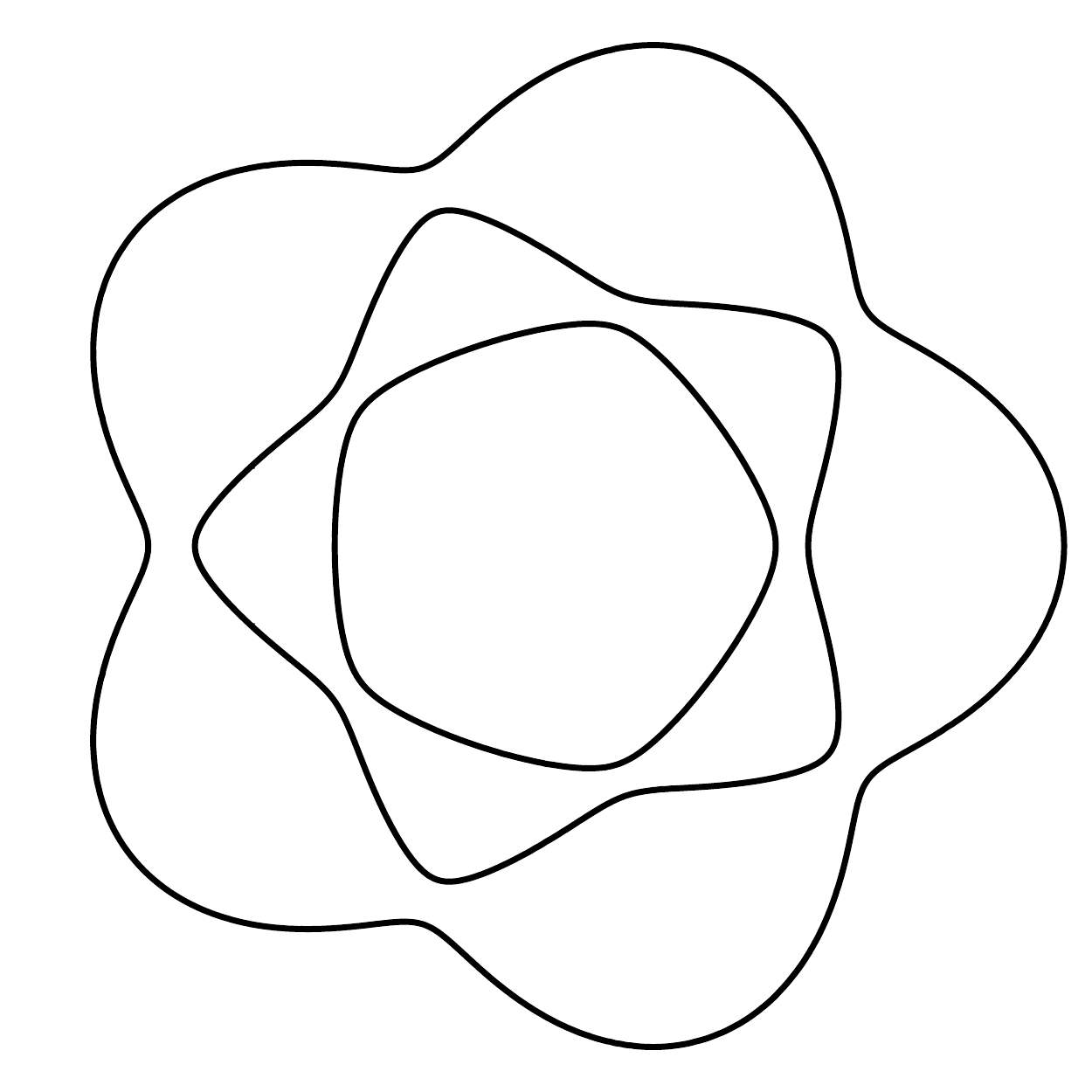} 
\includegraphics[height=1.5in]{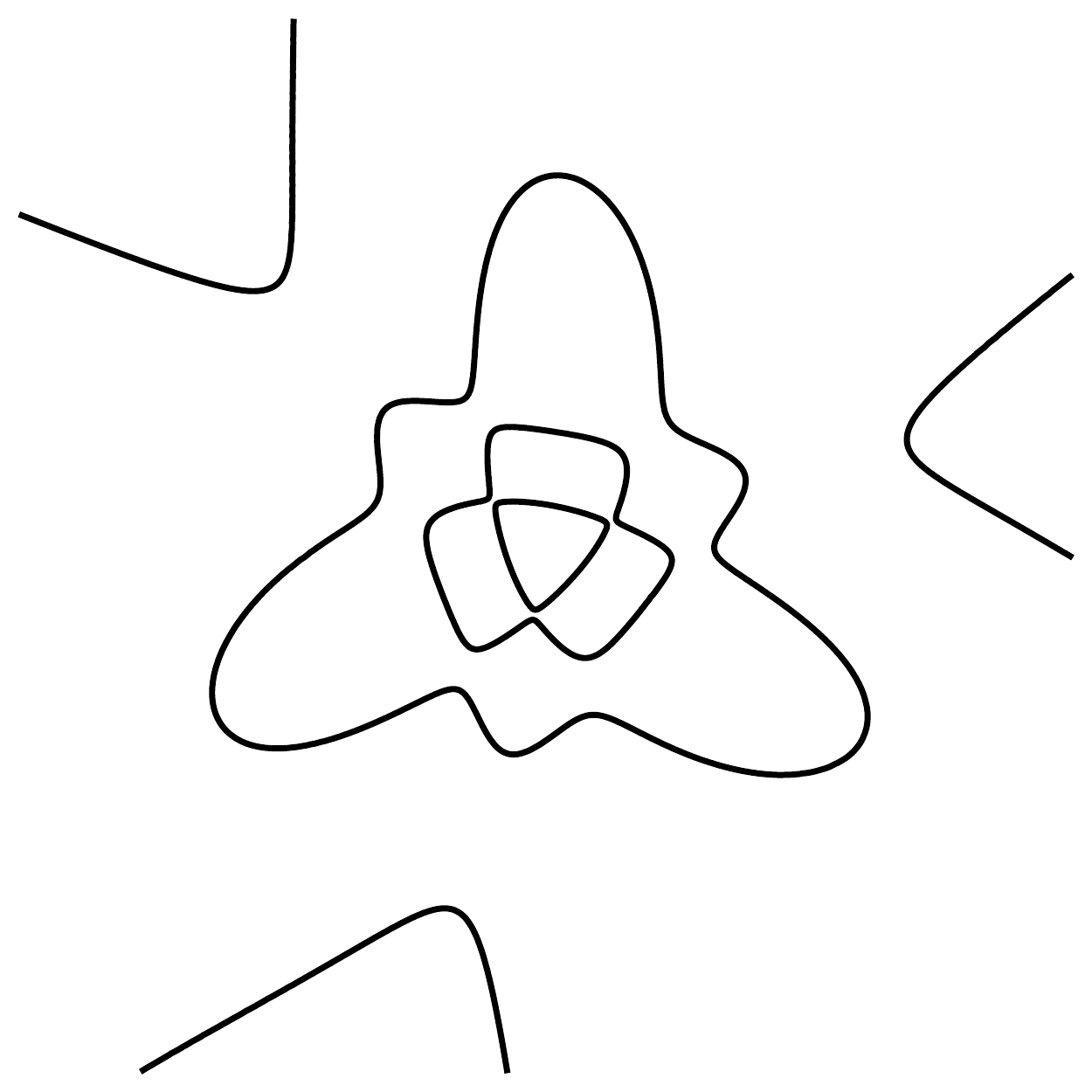}
\includegraphics[height=1.5in]{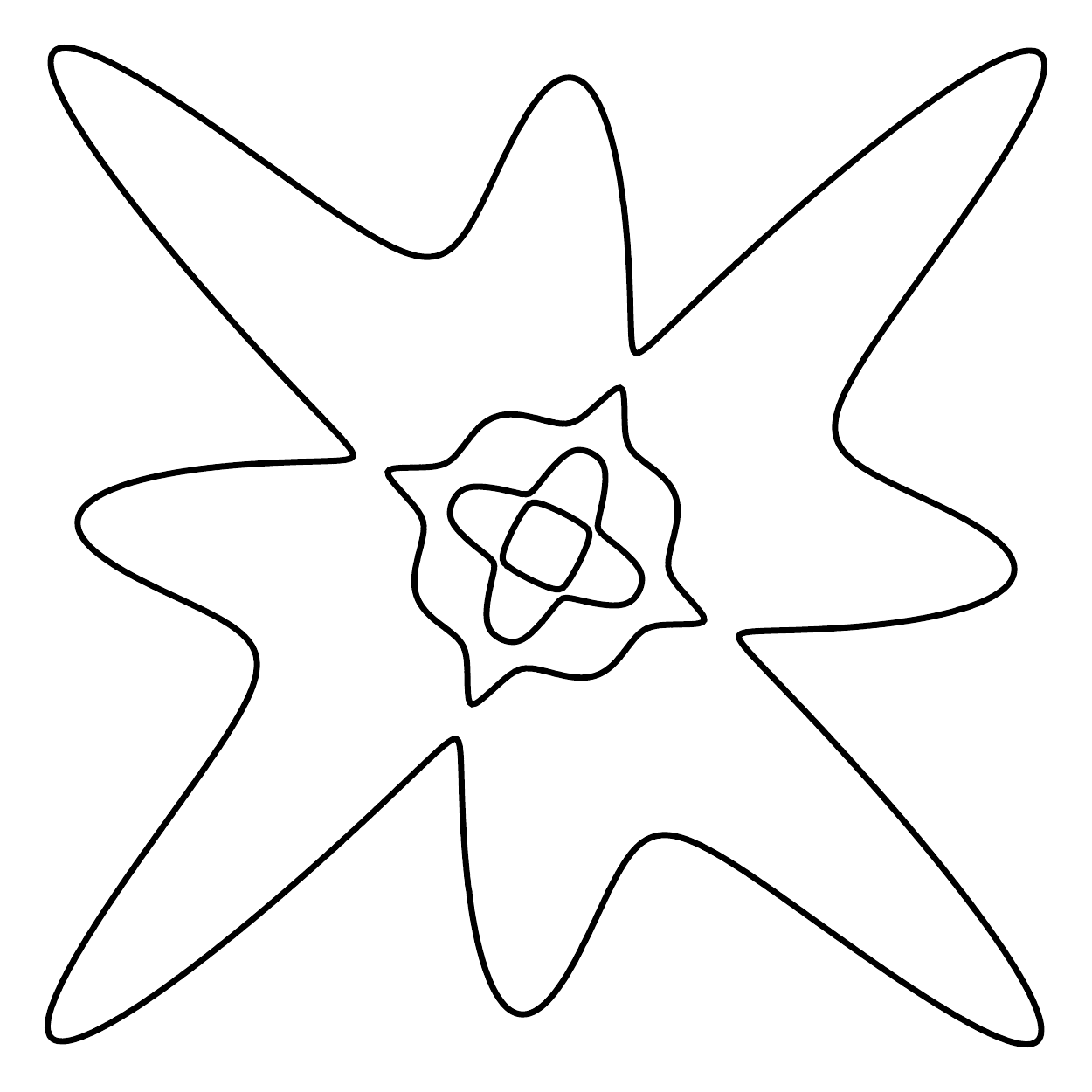} 
\caption{Curves in $\mathcal{H}_d^{C_n}$ for $(n,d) = (5,6)$, $(5,6)$, $(3,7)$, and $(4,8)$ (left to right).}\label{fig:gallery}
\end{figure}

\begin{proof}[Proof of Theorem~\ref{thm:InvHyp}]
We follow the proof of the main theorem in \cite{nuij}. 
Since strict hyperbolicity with respect to $(1,0,0)$ is an open condition on $\R[t,x,y]_d$, it suffices to 
show that $(\mathcal{H}^\circ)_d^{\Gamma}$ is non-empty. An explicit example is 
$t^{\delta}\cdot \prod_{i=1}^{D} (t^2-r_i(x^2+y^2))$ where $D = \lfloor \frac{d}{2} \rfloor$, $\delta\in\{0,1\}$ depending on the parity of $d$, and 
$r_1< \hdots < r_D \in \R_+$. 

The set  $\mathcal{H}_d^{\Gamma} $ is closed in the hyperplane in $\R[t,x,y]_d^{\Gamma}$ of polynomials with coefficient of $t^d$ equal to one.  To see that it is the closure of $(\mathcal{H}^\circ)_d^{\Gamma}$, let $f\in \mathcal{H}_d^{\Gamma} $.
By Lemma~\ref{lem:T_operator}, for $s>0$,  $\Ts{d}(f)$ is strictly hyperbolic with respect to $(1,0,0)$, meaning that 
$\Ts{d}(f)$ belongs to $(\mathcal{H}^\circ)_d^{\Gamma}$.  
The limit at $s=0$ is exactly $f$.
For $s \in \R$, consider the linear map $G_s:\R[t,x,y]_d \to \R[t,x,y]_d$ given by $ G_sf(t,x,y) = f(t,s^2x,s^2y).$ This map preserves hyperbolicity and invariance for any $s \in \R$ as well as strict hyperbolicity when $s\neq0$. 

For  $f \in \mathcal{H}_d^\Gamma$, consider the path in $\R[t,x,y]_d^{\Gamma}$ 
parametrized by $s\mapsto T_{1-s}^dG_sf$ for $s$ in $[0,1]$. 
At $s=1$, this gives $T_{0}^dG_1f = f$ and at $s=0$, this gives $T_{1}^dG_0f = T_{1}^dt^d$, 
which is independent of the choice of $f$. Note that for $s \in [0,1)$, we have 
$T_{1-s}^dG_sf\in \left(\mathcal{H}^\circ\right)_d^\Gamma$.  
The map $[0,1]\times \R[t,x,y]_d\to \R[t,x,y]_d$ given by $(s,f) \mapsto  T_{1-s}^dG_sf$ in the Euclidean topology defines a deformation retraction of both $\mathcal{H}_d^{\Gamma}$ and $(\mathcal{H}^\circ)_d^{\Gamma}$ onto the point $T_{1}^dt^d$. 
\end{proof}

\section{Cyclic weighted shift matrices and invariance}\label{sec:CWS}

One way of producing hyperbolic polynomials that are invariant under the 
actions of $C_n$ or $D_{2n}$ is via cyclic weighted shift matrices. 

  \begin{defn}\label{defn: CWS}
We call $A \in \F^{d \times d}$ a \hi{block cyclic weighted shift matrix  of order n} if
	$ A_{ij} = 0 \text{ if } j-i \neq 1 \mod n. $ 
	Let  $\CWS_{\mathbb{F}}(n,d)$ denote the set of such matrices. 
\end{defn}

\begin{remark}\label{rem:permutation}
The term ``block cyclic weighted shift matrix'' is justified after a permutation of the rows and columns of $A$. 
Consider the permutation of $[d]$ that groups numbers by their image modulo $n$ 
and otherwise keeps them in order. 
For example, for $n=3, d=5$, we consider the permutation $(1,2,3,4,5) \mapsto (1,4,2,5,3)$. 
After this permutation of rows and columns, a block cyclic weighted shift matrix is a block matrix 
consisting blocks of size $r\times s$ where $r,s\in \{\lfloor \frac{d}{n} \rfloor, \lceil \frac{d}{n} \rceil \}$, 
indexed by pairs $(i,j)$ of equivalence classes modulo $n$, where the block corresponding to $(i,j)$ 
is the zero matrix whenever $j-i\neq 1$ modulo $n$. 
\end{remark}
{\begin{example} An arbitrary matrix $A\in \CWS_{\F}(3,5)$ has the form
	\[A =  \begin{pmatrix}
		0 & a_{12} & 0 & 0 & a_{15}\\
		0 & 0 & a_{23} & 0 & 0 \\
		a_{31} & 0 & 0 & a_{34} & 0 \\
		0 & a_{42}& 0 & 0 & a_{45} \\
		0 & 0 & a_{53} & 0 & 0 
	\end{pmatrix}
	\text{ and } \ \ PAP^T = 
	\left(\!\!\begin{array}{cc|cc|c}
	0& 0& a_{12} & a_{15} & 0\\0& 0& a_{42}& a_{45}& 0\\ \hline 0& 0& 0& 0& a_{23}\\0& 
  0& 0& 0&a_{53} \\ \hline a_{31}& a_{34}& 0& 0& 0
	\end{array}\!\!\right)
	\] 
	where $P$ is the permutation matrix representing $(1,2,3,4,5) \mapsto (1,4,2,5,3)$. 
	The curve $\V_{\R}(F_A(1,x,y))$ and numerical range $\W(A)$ for such a matrix 
	are shown in Figure~\ref{fig:numRange}.
\end{example}}

The set of matrices $\CWS_{\C}(n,n)$, also called cyclic weighted shift matrices, have been studied extensively especially with respect to their numerical range \cite{Chien2013B, GTW2013, stout}. In general, the numerical range of any 
 matrix in $\CWS_{\C}(n,d)$ is invariant under multiplication by $n$th roots of unity.
To see this, define the group homomorphism $\rho: C_n \to \GL(\C^d)$ by 
\begin{equation} \label{eq: Omega}
\rho(\rot) = \Omega^* \  \text{ where } \ \Omega:= \diag\big(1, \omega, \omega^2, \ldots, \omega^d\big) \  \text{ and } \  \omega = e^{2\pi i/n}.
\end{equation} 
This induces an action of the cyclic group on $d\times d$ matrices 
by $\rot \cdot A = \Omega^* A \Omega$.  
Note that the $(i,j)$th entry of $\Omega^* A \Omega$ is $\omega^{j-i}A_{ij}$. 
Therefore, for any matrix $A\in \CWS_{\C}(n,d)$, 
the cyclic group acts by scaling by the $n$th root of unity. That is, 
$$\Omega^* A \Omega \ = \  \omega A.$$ 
Since the matrix $\Omega$ is unitary and numerical ranges are invariant under 
conjugation by unitary matrices, we see that the numerical range of 
$A$ is invariant under multiplication by $n$th roots of unity, i.e.~$\W(A) = \W(\omega A)$. 
Chien and Nakazato \cite{Chien2013} show that for matrices $A\in \CWS_{\mathbb{F}}(n,n)$, the 
polynomials $F_A$ are invariant under the cyclic group (for $\mathbb{F}=\C$) and dihedral group (for $\mathbb{F}=\R$). 
Here we generalize this observation to matrices of arbitrary size. 

To do this it is useful to rewrite the polynomial $F_A$ as
\[
F_A(t,x,y)  \ = \ \det\left(t I + \tfrac{1}{2}(x+iy)A^*+ \tfrac{1}{2}(x-iy)A\right).
\]
This is particular convenient as the group $C_n$ acts diagonally 
on the linear forms $t$, $x+iy$, and $x-iy$. Specifically, each action fixes $t$ and we have
\begin{align*}
\rot\cdot(x+iy) & =  e^{-2\pi i /n}(x+iy) & \refl\cdot(x+iy) & = x-iy\\
\rot\cdot(x-iy) & = \  e^{2\pi i /n}(x-iy) & \refl\cdot(x-iy) & = x+iy.
\end{align*}

\begin{prop}\label{prop:CWS=>invariant} 
For any $A\in \C^{d \times d}$, the polynomial $F_A$ is hyperbolic with respect to $(1,0,0)$. 
If $A\in \CWS_{\C}(n,d)$, then $F_A $ belongs to $\HH^{C_n}_d$ and if $A\in \CWS_{\R}(n,d)$, then $F_A$ belongs to $ \HH^{D_{2n}}_d$.
\end{prop}

Here $\HH^{\Gamma}_d$ denotes the set of invariant hyperbolic polynomials as in Section~\ref{sec:Nuij}. 

\begin{proof} 
By definition, the polynomial $F_A$ is the determinant of a linear matrix pencil that equals the identity matrix  
at $(1,0,0)$. The hyperbolicity of $F_A$ then follows from the fact that all of the eigenvalues are real. 
For invariance, it suffices to check that $\rot \cdot F_A=F_A$ for $A\in \CWS_{\C}(n,d)$ and 
$\refl \cdot F_A=F_A$ for $A\in \CWS_{\R}(n,d)$. Following \cite{Chien2013} and using the simplification of $F_A$ above,  
we apply rotation to give 
\begin{align*}
\rot \cdot F_A(t,x,y) 
&= \det\left(t I + \tfrac{1}{2}(x+iy)(\omega A)^*+ \tfrac{1}{2}(x-iy)(\omega A)  \right)\\
&= \det \left( tI  +\tfrac{1}{2}(x+iy)(\Omega^* A \Omega)^* +\tfrac{1}{2}(x-iy)(\Omega^* A \Omega)\right)  \\
 &= \det(\Omega^*)\cdot \det \left( tI + \tfrac{1}{2}(x+iy)A^*+ \tfrac{1}{2}(x-iy) A  \right) \cdot\det(\Omega)  = F_A(t,x,y).
\end{align*}
Similarly, if $A$ has real entries then  $A^* = A^T$ and
\begin{align*}
	\refl \cdot F_A(t,x,y) = F_A(t,x,-y) 
  	&= \det\left(t I  + \tfrac{1}{2}(x-iy)A^T+ \tfrac{1}{2}(x+iy)A\right)  \\
    &= \det\left(\left(t I + \tfrac{1}{2}(x+iy)A^T+ \tfrac{1}{2}(x-iy)A \right)^T\right) =F_A(t,x,y).
\end{align*}
\end{proof}
Chien and Nakazato asked the converse question and provided a positive answer for the case when $d=n=3,4$. The authors of \cite{patton2013, patton2011} studied rotational symmetry of the numerical range of matrices of size $d=3,4$. 
We will provide a converse in the case $d=qn$ in Theorem~\ref{thm: num range}. 
The difficulty for arbitrary $d$ comes from the fact that for many values of $d$, all forms in
$(\HH^\circ)_d^{\Gamma}$ define curves with complex singularities, as discussed in Section~\ref{sec:questions}.

\section{A constructive proof for smooth curves}
\label{sec: d=qn smooth}

In this section we aim to prove Theorem~\ref{thm:d=qn}, but with some added assumptions about of a curve in $\HH_d^{\Gamma}$ and an interlacer. Throughout Sections~\ref{sec: d=qn smooth} and \ref{sec: d=qn dihedral} we will assume that
	\begin{enumerate}
		\item[\textit{A1.}] $f \in (\HH^\circ)_d^{\Gamma}$  and $\V_{\C}(f)$ is smooth,
		\item[\textit{A2.}] $g \in \HH_{d-1}^{\Gamma}$ interlaces $f$ with respect to $(1,0,0)$,
		\item[\textit{A3.}] $\V_{\C}(f)$ and $\V_{\C}(g)$ intersect transversely, and
		\item[\textit{A4.}] $\big|\V_{\C}(f,g,t)\big| =
										\begin{cases}
										0 &\text{if $n$ is odd,}\\
										d &\text{if $n$ is even.}
										\end{cases}$
	\end{enumerate}
Specifically, we prove the following theorem.
 \begin{thm}\label{thm: d=qn transverse}
	Let $d=qn$ for some $q \in \Z_{+}$. Let $f$ and $g$ satisfy (A1)--(A4). 
	\begin{itemize}
		\item[(a)] If $\Gamma = C_n$, then there exists a  matrix $A \in \CWS_{\C}(n,d)$ so that $f=F_A$.
		\item[(b)] If $\Gamma = D_{2n}$, then there exists a  matrix $A \in \CWS_{\R}(n,d)$ so that $f=F_A$.
	\end{itemize}
\end{thm}

In order to construct the matrix $A$ and corresponding determinantal representation of $f$, 
we first construct the adjugate of this matrix, which will be a $d\times d$ matrix of forms of degree $d-1$ 
that has rank $\leq 1$ on the curve $\V_{\C}(f)$. Following \cite{Dixon, LP, PV},
we take $g$ to be the $(1,1)$ entry of this adjugate matrix and fill in the first row and column 
to vanish on complimentary sets of points in $\V_{\C}(f)$ and $\V_{\C}(g)$.

The next lemma is a general statement about complex points in the intersection of $\V_{\C}(f)$ and $\V_{\C}(g)$. 
This will allow us to split these intersection points into disjoint sets determined by orbits under the action of rotation.
\begin{lem}\label{lem: xy ratio real}
Let $f$ and $g$ satisfy Assumptions (A1)--(A4). Any point $[t:x:y]$ in $\V_{\C}(f,g)$  
satisfies $|x+iy| \neq |x-iy|$.
\end{lem}

\begin{proof}
For the sake of contradiction, suppose that $|x+iy| = |x-iy|$. If $y\neq 0$, 
then $x+iy = z(x-iy)$ with $|z|=1$ and $z\neq 1$. Solving for $x/y$ gives 
	\[
\frac{x}{y} \ =  \frac{-i(1+z)}{1-z} \cdot\frac{(1-\overline{z})}{(1-\overline{z})} \ = \  \frac{-i(1-z+\overline{z} -z\overline{z})}{| 1-z|^2} \ = \  \frac{-i(-z+\overline{z})}{| 1-z|^2} = \  \frac{2 \cdot {\rm Im}(z)}{| 1-z|^2} \ \in\  \R.
\]
By homogeneity of $f$, $f(t/y,x/y,1)=0$, meaning that $t/y$ is a root of the polynomial $f(\lambda,x/y,1)\in \R[\lambda]$ where $x/y \in \R$ is fixed. The hyperbolicity of $f$ then implies that $t/y \in \R$.  Since both $x/y$ and $t/y$ are real,
the point $[t:x:y]$ belongs to $\P^2(\R)$. 
By \cite[Proposition~4.3]{PV}, any real intersection point of $\V_{\C}(f)$ and $\V_{\C}(g)$ is non-transverse, contradicting (A3). 

Similarly, if $y = 0$, then $x\neq 0$ since $f(1,0,0) = 1$. Then $f(t/x,1,0)=0$, implying that $t/x\in \R$ and 
$[t:x:y]=[t:x:0]$ belongs to $\P^2(\R)$, again contradicting (A3). \end{proof}

\begin{cor} \label{cor: orbit and xyconj orbit distinct}
	Let $f$ and $g$ satisfy Assumptions (A1)--(A4). Then each $C_n$-orbit in $\V_{\C}(f,g)$ is disjoint from its image under conjugation.
\end{cor}

\begin{proof}
Let $\OO$ be a $C_n$-orbit of points in $\V_{\C}(f,g)$ and suppose $p = [t:x:y] \in \OO\cap \overline{\OO}$. 
Then  $\rot^{\ell}\cdot p= \overline{p}$ for some $\ell \in [n]$.
If $\rot^{\ell}\cdot p=[t:a:b] $, then $[t:\omega^{-\ell}(x+iy): \omega^{\ell}(x-iy)]=[t:a+ib:a-ib]$.
In particular, if $\rot^{\ell}\cdot p$ equals $\overline{p}$,  then 
\begin{equation}\label{eq: lem xy ratio real}
\left[t:\omega^{-\ell}\left(x+iy\right):\omega^{\ell}\left(x-iy\right)\right]
=
[\overline{t}:\overline{x}+i\overline{y}:\overline{x}-i\overline{y}] 
=
[\overline{t}:\overline{x-iy}:\overline{x+iy}].
\end{equation} 
The cross ratio the the last two coordinates gives 
$\overline{x-iy}\cdot \omega^{\ell}\left(x-iy\right) = \overline{x+iy}\cdot \omega^{-\ell}\left(x+iy\right)$.
Taking the modulus of both sides shows that $|x-iy| = |x+iy|$, contradicting Lemma~\ref{lem: xy ratio real}. 
\end{proof}

Define the linear map 
	\begin{equation} \label{eq: phimap} \varphi : \C[t,x,y] \to \C[t,x,y] \ \  \text{ where }  \ \ h(t,x,y) \mapsto h(\rot^{-1}\cdot (t,x,y)).
	\end{equation}
	The eigenvectors of this map have the form $t^l(x+iy)^j(x-iy)^k$, each with eigenvalue $\omega^{j-k}$. 
	The restriction $\varphi|_d$ of $\varphi$ to $\C[t,x,y]_d$ has a finite number of eigenvectors equal to $\dim_{\C}(\C[t,x,y]_d) = \binom{d+2}{2}$. For  each $\ell = 0, 1, \ldots, n-1$, denote by 
	\begin{equation} \label{eq: esl}
		\esl_d = \left\{ f \in \C[t,x,y]_d : \varphi(f) = \omega^{\ell} f\right\}
	\end{equation}
	the eigenspace of the restriction $\varphi |_d$ associated to eigenvalue $\omega^{\ell}$. Notice $\es{0}_d = \C[t,x,y]_d^{C_n}$ and we can write $\C[t,x,y]_d$ as a decomposition of eigenspaces 
$$\C[t,x,y]_d = \bigoplus_{\ell=0}^{n-1} \es{\ell}_{d}.$$ 
We will be interested in  the dimension of each eigenspace. In particular, for $d=qn$, we want at least $q$ elements in each eigenspace $\es{\ell}_{d-1}$ in order to choose linearly independent set of elements in $\C[t,x,y]_{d-1}$ for the first row of the adjugate matrix we wish to construct.

\begin{lem} \label{lem: dim eigenspaces d=qn}
	Let $d=qn$ for some $q \in \Z_{+}$. The dimension of the eigenspace $\es{\ell}_{d-1}$  is  
	$$\dim_{\C}\left(\es{\ell}_{d-1}\right) = \begin{cases} 	
\frac{dq}{2}+\frac{q}{2} &\text{if $n$ is odd}	\\
\frac{dq}{2}+q &\text{if $n$ is even and $\ell$ is odd}\\
\frac{dq}{2} &\text{if $n$ is even and $\ell$ is even}.
	\end{cases}$$
\end{lem}

\begin{proof} 
The monomials $t^{d-1-j-k}(x+iy)^j(x-iy)^k$ where $j-k\equiv \ell \mod n$ form a basis 
for the vectorspace  $\left(\es{\ell}_{d-1}\right)$.
Thus the dimension of $\left(\es{\ell}_{d-1}\right)$ is the number points in 
the simplex $\{(j,k)\in \Z_{\geq 0}^2 : j+k\leq d-1\}$ with  $j-k\equiv \ell \mod n$.
Note that the such first points on the $j$ and $k$ axes will be $(\ell,0)$ and $(0,n-\ell)$. 

For any $0\leq j\leq d-1$, the number of integer points of the form $(a+j,a)$ in this simplex  
is given by $\lceil \frac{d-j}{2} \rceil$. Similarly the number of 
integer points of the form $(a,a+k)$ is $\lceil \frac{d-k}{2} \rceil$. 
We are interested in  these values when $j=\ell+an$ and $k = n-\ell+an$. That is, 
\[
\dim\left(\es{\ell}_{d-1}\right)  = \sum_{a=0}^{q-1}\left\lceil\frac{d-(\ell+an)}{2}\right\rceil +\left\lceil\frac{d-(n-\ell+an)}{2}\right\rceil.
\]
When $n$ is odd, $\ell$ and $n-\ell$ have different parities, meaning exactly one of 
$\frac{d-(\ell+an)}{2}$ and $\frac{d-(n-\ell+an)}{2}$ will be an integer. 
When $n$ is even the parities of $d-(\ell+an)$ and $d-(n-\ell+an)$ depend only on the parity of 
$\ell$. They are odd if and only if $\ell$ is odd. 
Let $\delta = 1$ when $n$ is odd, $2$ when $n$ is even and $\ell$ is odd, and zero otherwise. 
Then  $\dim\left(\es{\ell}_{d-1}\right) $ equals 
\[
\sum_{a=0}^{q-1}\frac{d-(\ell+an)+d-(n-\ell+an)+\delta}{2} 
= \sum_{a=0}^{q-1}\frac{2d-n+\delta-2an}{2}
= q \frac{2d-n+\delta}{2} - n\frac{q(q-1)}{2},
 \]
where the last equality is obtained by summing the arithmetic sequence. 
Recalling that $d=qn$, we see that this dimension simplifies to $q (q n + \delta)/2 = q(d+ \delta)/2$, as desired. 
\end{proof}

%

For $f$ and $g$ that satisfy (A1)--(A4), we will  split the points of $\V_\C(f,g)$ into $S \cup \overline{S}$ based on orbits under rotation. The next lemma helps enumerate conditions imposed by the set of orbit representatives, and accurately count dimensions later in Lemma~\ref{lem: d=qn cyclic basis}.

\begin{lem}\label{lem: factor of t}
	Let $d=qn$ for some $q \in \Z_{+}$. If $n$ and $\ell$ are even, 
	then each monomial in $\es{\ell}_{d-1}$ has a factor of $t$.
\end{lem}	

\begin{proof}
	Let $t^{d-1-j-k}(x+iy)^j(x-iy)^k$ be an arbitrary monomial in $\esl_{d-1}$. Then $j-k \equiv \ell \mod n$. 
	Since $n$ and $\ell$ are even, $j-k$ is even and so is $j+k=j-k+2k$.
	Moreover $d = qn$ is even and $d-1$ is odd. It follows that 	
       the exponent $d-1-(j+k)$ of $t$ is odd and $ \geq 1$.
\end{proof}

By Corollary~\ref{cor: orbit and xyconj orbit distinct}, $\V_{\C}(f,g)$ may be split into two disjoint sets according to orbits invariant under the action of $C_n$. More explicitly, write $\V_{\C}(f,g) = S \cup \overline{S}$ as the union of two disjoint conjugate sets. Define $\tilde{S}$ to be a minimal set of orbit representatives from $S$ so that 
	\begin{equation} \label{S}
		S = \left\{ \text{rot}^{\ell} \cdot p \mid p \in \tilde{S}, \ell \in [n] \right\}\text{.}
	\end{equation}	
The next proposition gives the maximum number of possible conditions imposed by $\tilde{S}$ on an element of $\es{\ell}_{d-1}$.
\begin{prop}\label{prop: tildeS}
	Let $d = qn$ for some $q \in \Z_{+}$ and suppose $f$ and $g$ satisfy (A1)--(A4). Then the number of distinct orbits in $S$ is  \[
	\big|\tilde{S}\big| = \begin{cases}
		q(d-1)/2 &\text{ if $n$ is odd}\\
		qd/2 &\text{ if $n$ is even.}
	\end{cases}
	\]
\end{prop}

\begin{proof}
Since $f(1,0,0)\neq 0$, each point $[t:x:y]\in \V_{\C}(f,g)$ with $t\neq 0$ generates a $C_n$-orbit of size $n$,  
	since $\rot^{\ell}$ fixes such a point if and only if $\ell \equiv 0 \mod n$. 
When $n$ is odd, all points in $ \V_{\C}(f,g)$ have $t\neq 0$, 
so the $d(d-1)/2$ total points of $S$ split up into $d(d-1)/2n = q(d-1)/2$ orbits under $C_n$. 
When $n$ is even, the $d(d-2)/2$ total points in $S$ with $t\neq 0$ split up into $q(d-2)/2$ orbits. 
Each point in $\V_{\C}(f,g,t)$ generates a $C_n$-orbit of size $n/2$ since $\rot^{n/2}$ acts as the identity on points of the form $[0:x:y]$. Thus the $d$ total points of $\V_{\C}(f,g,t)$ contribute $(d/2)/(n/2)=q$ orbits to $S$ which means $S$ has a total of $qd/2$ $C_n$-orbits.
\end{proof}

Denote the space of forms in $\C[t,x,y]_{d-1}$ vanishing on  points $S$ and $\tilde{S}$ from (\ref{S}) by 
	\begin{equation} \label{eq: idealS}
			\mathcal{I}\left(S\right)_{d-1}  \ \text{ and } \ \ \mathcal{I}(\tilde{S})_{d-1}
	\end{equation} 
respectively. Now we can show there are enough elements in each eigenspace to choose a linearly independent forms in $\C[t,x,y]_{d-1}$ for the first row in our desired adjugate matrix.
\begin{lem}\label{lem: d=qn cyclic basis}
	Let $d=qn$ for some $q \in \Z_{+}$. There exist $q$ linearly independent polynomials  in each eigenspace $\es{\ell}_{d-1}$ that vanish on the points $S$. That is, \[\dim_{\C}\left( \es{\ell}_{d-1} \cap \mathcal{I}(S)_{d-1}\right) \geq q.\]
\end{lem}

\begin{proof}
An element of $\es{\ell}_{d-1}$ vanishes on $S$ if and only if it vanishes on $\tilde{S}$. 
	Then for any $\ell$,
	\[
	\dim\left( \es{\ell}_{d-1} \cap \mathcal{I}(S)_{d-1}\right) = \dim\left( \es{\ell}_{d-1} \cap \mathcal{I}\big(\tilde{S}\big)_{d-1}\right) \\
	\geq \dim\left( \es{\ell}_{d-1}\right) - \left|\tilde{S}\right|.
	\]
	
 In the cases when $n$ is odd or $n$ is even with $\ell$ odd, this count is straightforward due to Lemmas~\ref{lem: dim eigenspaces d=qn} and \ref{prop: tildeS}. By Lemma~\ref{lem: factor of t}, when $n$ and $\ell$ are even, every monomial in $\esl_{d-1}$ has a factor of $t$. Thus every element of $\esl_{d-1}$ will already vanish at points with $t=0$ without adding additional constraints from those in $\V_{\C}(f,g,t)$. In this case we do not take into account the $q$ orbits at infinity and using Lemmas~\ref{lem: dim eigenspaces d=qn} and \ref{prop: tildeS} we have 
 \[
 \dim\left( \es{\ell}_{d-1} \cap \mathcal{I}(S)_{d-1}\right) \geq \dim\left( \es{\ell}_{d-1}\right) - \left|\tilde{S}\right| + q \geq q.
 \]
\end{proof}

The final piece to the construction is an invariant version of 
Max Noether's Theorem on divisors on smooth plane curves, appearing in \cite{LP}.

\begin{lem}[Lemma 3.7~\cite{LP}]\label{lem: maxnoether}
Suppose $f, g\in \es{0}$ and $h\in \es{\ell}$ are homogeneous with $\V_{\C}(f)$ smooth where $\deg(h) > \deg(f),\deg(g)$ and $g$ and $h$ have no irreducible components in common with $f$.
 If $\V_{\C}(f,g)$ consists of distinct points and  $\V_{\C}(f,g)\subseteq  \V_{\C}(f,h)$, then there exists homogeneous $a, b\in \es{\ell}$ with $\deg(a)= \deg(h)-\deg(f)$ and $\deg(b)= \deg(h) -\deg(g)$ so that $h = af+bg$. Moreover, if $f$, $g$, and $h$ are real, then $a$ and $b$ can be chosen real.
\end{lem}

The construction below is similar to Construction~4.1 from \cite{LP}. For normalization of the coefficient matrix of $t$, however, we must be more careful. Now the variable $t$ appears in off-diagonal entries of the determinantal representation, so we must first block diagonalize the coefficient matrix of $t$, then normalize with respect to each block separately in order to preserve the desired matrix structure.

\begin{construction}\label{con: d=qn transverse} {\rm Let $d=qn$ for some $q \in \Z_+$ and $\Gamma = C_{n}$.\\
Input: Two plane curves $f$ and $g$ satisfying (A1)--(A4).\\
Output:   $A \in \CWS_{\C}(n,d)$ with $f=F_A$.
	\begin{enumerate}
		\item Set $g_{11} = g$.
		\item Split up the distinct $d(d-1)$ points of $\V_{\C}(f,g_{11})$ into two disjoint, conjugate sets  $S \cup \overline{S}$ of $C_n$-orbits such that $ \rot(S)=S$.
		\item Extend $g_{11}$ to a linearly independent set $\{g_{11}, g_{12}, \ldots, g_{1d}\} \subset \C[t,x,y]_{d-1}$ vanishing on all points of $S$ with $g_{1j} \in \es{1-j}_{d-1}$ for all $j \in [n]$ and set 
		$g_{j1} = \overline{g_{1j}}$ for each $j$.
		\item For $1<i\leq j$, choose $g_{ij} \in \es{i-j}_{d-1}$ so that $g_{11}g_{ij}-\overline{g_{1i}}g_{1j} \in \langle f \rangle$ and $g_{ii} \in \R[t,x,y]$.
		\item For $i < j$, set $g_{ji} = \overline{g_{ij}}$ and define $G = (g_{ij})_{i,j} \in (\C[t,x,y]_{d-1})^{d \times d}$.
		\item Define $M = (1/f^{d-2})\cdot \adj{G}$.
		\item For $\ell \in [d]$, write $\ell-1=an+b$ for some integers $a$ and $b$ with $0 \leq b \leq n-1$. Let $P$ be the permutation matrix that takes $\ell=an+b+1$ to $bq+a+1$. Define $M'=PMP^T$ as a matrix with $q \times q$ blocks $M'=(M'_{kl})_{k,l=1}^n$ 
		\[
		\left(M'_{kl}\right)_{ij} \in \es{k-l}_1 \text{ for } k,l \in [n] \text{ and }  i,j \in [q].
		\]
		\item For each $k$ compute the Cholesky decomposition of each diagonal block $M'_{kk}$ and write $(M')^{-1}_{kk}= U_{k}U^{\ast}_{k}$ for some $U_k \in \C^{q \times q}$.
		\item Define $U = \diag(U_1,U_2,\ldots,U_k)$ and output $A=\left(P^{T}U^{\ast}M'UP\right)(0,1,i)$.
	\end{enumerate}}
\end{construction}

 \begin{proof}[Proof of Theorem~\ref{thm: d=qn transverse}]
            Our goal is to show each step of Construction~\ref{con: d=qn transverse} can be completed and 
            produces a matrix $A \in \CWS_{\C}(n,d)$ such that $f = F_A$ as in \eqref{eq:FA}.
            Let $g_{11} = g$. By Corollary~\ref{cor: orbit and xyconj orbit distinct}, we can write 
            $\mathcal{V}_{\C}(f,g_{11})$ as a disjoint union $S \cup \overline{S}$ where $\rot(S) = S$. 
            For Step 3, Lemma~\ref{lem: dim eigenspaces d=qn} 
            allows us extend $g_{11}$ to a linearly independent set $\{g_{11}, g_{12}, \hdots, g_{1d}\}$
            where  $g_{1j} \in \es{1-j}_{d-1}$  vanishes on $S$ for every $j \in [n]$. Now let $g_{j1} = \overline{g_{1j}}$.
            By Lemma~\ref{lem: maxnoether}, we can choose 
            $g_{ij}$ such that $g_{ij} \in \es{i-j}_{d-1}$ for $1 < i < j$
            and $g_{11}g_{ij}-\overline{g_{1i}}g_{1j} =af$ for some homogeneous $a \in \mathbb{C}[t,x,y]$ 
            to complete Step 4. Since $f$, $g_{11}$, $\overline{g_{1i}}g_{1i} \in \R[t,x,y]$, we can choose 
            $g_{ii} \in \R[t,x,y]$ as well. Let $g_{ji} = \overline{g_{ij}}$ for $i < j$ and define $G = (g_{ij})_{i,j}$ be the 
            $d \times d$ complex matrix of forms of degree $d-1$. 
            By Theorem~$4.6$ of \cite{PV}, each entry of $\adj{G}$ 
will be divisible by $f^{d-2}$ and Step 6 is valid. 
The entries in $G$ have degree $d-1$, so entries of its adjugate have degree $(d-1)^2$. Then $f^{d-2}$ has degree $d(d-2)$, so entries of $M$ are linear in $t, x,$ and $y$. 
By  \cite[Theorem~$4.6$]{PV}, $M(1,0,0)$ is positive definite and $\det(M)$ is a nonzero scalar multiple of $f$. 
Let $\Omega = \diag(1, \omega, \ldots, \omega^{d-1})$.
Applying the map $\varphi$ to the $(i,j)$-th entry of $M = (m_{ij})_{ij}$ gives
 \begin{align*}
            \varphi\left(m_{ij}\right) 
            &= (1/f^{d-2})\cdot \adj{{\varphi}(G)}_{ij}\\
            &= (1/f^{d-2})\cdot \adj{\Omega G\Omega^{\ast}}_{ij}\\
            &= (1/f^{d-2}) \cdot(\adj{\Omega^{\ast}}\adj{G}\adj{\Omega)}_{ij}  \\
            &= (\Omega M\Omega^{\ast})_{ij}\\
            &= \overline{\omega^{j-1}}\omega^{i-1}m_{ij} \\
            &=\omega^{i-j}m_{ij}\text{.}
		\end{align*}
		
Therefore, $m_{ij} \in \es{i-j}_{1}$ for each $i,j$. 
The restriction of $\varphi$ to $\C[t,x,y]_1$ has eigenvalues $1, \omega$, and $\omega^{n-1}$ with associated eigenspaces $\es{0}_{1}$, $\es{1}_{1}$, and $\es{n-1}_{1}$. This implies $m_{ij}=0$ if $i-j \not\equiv 0, \pm1\mod n$. For $m_{ij}$ such that $i-j \equiv n-1 \mod n$, we have $m_{ij}\in \es{n-1}_1$, showing that $m_{ij}$ is a scalar multiple of $x-iy$. Similarly, since $M$ is Hermitian, this implies $m_{ji} \in \es{1}_1$ must be a scalar multiples of $x+iy$. 
 If $i-j \equiv 0 \mod n$ then $i\equiv j \mod n$ and $m_{ij} \in \es{0}\cap \R[t,x,y]$ is a multiple of $t$.
\par Next we will show by permuting rows and columns of $M$ we may get the identity matrix as the coefficient of $t$ in our representation. Consider $M$ as a matrix of $n \times n$ blocks. Each block is a cyclic weighted shift matrix and there are $q^2$ blocks in total. For $\ell \in [d]$, write $\ell-1=an+b$ for some integers $a$ and $b$ with $0 \leq b \leq n-1$. Let $P$ be the permutation matrix that takes $\ell=an+b+1$ to $bq+a+1$, as in Remark~\ref{rem:permutation}. Define $M'=PMP^T$ as a matrix with $q \times q$ blocks $M'=(M'_{kl})_{k,l=1}^n$ with
		\[
		 \left(M'_{kl}\right)_{ij} \in \es{k-l}_1 \text{ for } k,l \in [n] \text{ and }  i,j \in [q].
		\] 
It follows that $M'(1,0,0)$ is a block diagonal matrix. Moreover, since $M(1,0,0)$ is positive definite, so is $M'(1,0,0)$. For each $k \in [n]$ we can decompose $M'_{kk}(1,0,0)$ so that $M'_{kk}(1,0,0)^{-1} =U^{\ast}_k U_k$ for some $U_k~\in~\C^{q \times q}$. Define $U = \diag(U_1,U_2,\ldots,U_n)$. Then $M''=U M' U^{\ast}$ is a desired representation of $f$ since $M''(1,0,0) = I_d$,  $\left(M''_{kl}\right)_{ij} \in \es{k-l}_1$ for $k,l \in [n]$ and $i,j \in [q]$, and  $f = (1/\lambda) \cdot \det\left(U M' U^{\ast}\right) \text{ for } \lambda = \det(U)\cdot\det(U^{\ast})$. Lastly, apply the inverse permutation so $f = (1/\lambda)\cdot\det(P^T M'' P)$ and evaluating $(P^T M'' P)(0,1,i)$ gives a cyclic weighted shift matrix of order $n$.
		\end{proof}
		
\begin{example}\label{ex:Counts63}
[$d=6$, $n=3$, $q=2$]	
For  $n=3$ and $d=6$, we see that $\V_{\C}(f,g)$ consists of $d(d-1)  = 30$ points, 
which split into $10$ orbits, each of size $3$. 
These orbits come in conjugate pairs, of which we take half to 
form the set $S$, which will have size $15$. 
The set $\tilde{S}$ of orbit representatives in $S$ has size  $q(d-1)/2=5$.
For every $\ell$, $\es{\ell}_{d-1}$ has dimension $q(d+1)/2 = 7$. 
Since each point in $\tilde{S}$ imposes a linear condition on forms in $\es{\ell}_{d-1}$, 
we can find $7-5=2$ linearly independent forms $g_{1j},g_{1(j+n)} \in \es{1-j}$ that vanish on $\tilde{S}$.
Ranging over $j=0,1,2$ gives the first row of the matrix $G$. 
The entries of the linear matrix $M = (m_{ij})_{ij}$ satisfy $m_{ij}\in  \es{i-j}_{1}$. 
In particular, the $(i,j)$ entry of $M(1,0,0)$ zero whenever $i\neq j \mod 3$. 
Evaluating the matrix $M' = PMP^T$ at $(t,x,y) = (1,0,0)$ therefore results in a 
block diagonal matrix of three $2\times 2$ blocks, each of which is positive definite. 
Conjugation by the appropriate block diagonal matrix results in the identity matrix 
and the desired determinantal representation of $f$. 
A detailed example of this construction can be found in \cite[Example~3.1.9]{thesis}. 
\end{example}

\begin{example}[$d=12$, $n=4$, $q=3$]	
For  $n=4$ and $d=12$, we see that $\V_{\C}(f,g)$ consists of $d(d-1)  = 132$ points.
Of these,  $120$ have $t\neq 0$ and split up into $q(d-2) = 30$ orbits of size $4$.
Since $g$ has a factor of $t$, there are an additional $12$ points with $t=0$, 
splitting up into $6$ orbits, each of size two. 
Splitting these $132$ into conjugate pairs $S\cup \overline{S}$, 
we see that $S$ has $66$ points, consisting of $15$ orbits of size $4$ and 
three orbits of size two.  The set $\tilde{S}$ of orbit representatives has size $qd/2=18$.
The dimension of $\es{\ell}_{d-1}$ is $qd/2 = 18$ when 
$\ell$ is even and $qd/2+q = 21$ when $\ell$ is odd. 
Note that when $\ell$ is even, elements of $\es{\ell}_{d-1}$ have a factor of $t$ 
and so automatically vanish on those points in  $\tilde{S}$ with $t=0$. 
Each of the $15$ remaining points imposes a linear condition on $\es{\ell}_{d-1}$, 
leaving a three-dimensional subspace of forms in  $\es{\ell}_{d-1}$ that vanish on $S$. 
Similarly, if $\ell$ is odd, then $\dim_{\C}\es{\ell}_{d-1}-|\tilde{S}| = 21 - 18 = 3$. 
Therefore for each $j=0,1,2,3$, we can choose linearly independent 
$g_{1j}, g_{1(j+n)}, g_{1(j+2n)}$ in $\es{1-j}_{d-1}$ that vanish on $S$. 
\end{example}

\section{Dihedral Invariance}
\label{sec: d=qn dihedral}

In this section, we modify Construction~\ref{con: d=qn transverse} to include the invariance under reflection
and produce a matrix in $ \CWS_{\R}(n,d)$. We divide the points of $\V_\C(f,g)$ based on orbits under rotation, then split according to reflection. Specifically, we require not only that $\V_\C(f,g) = S \cup \overline{S}$ where $\rot(S) = S$, but also $\refl\left(\overline{S}\right) = S$ meaning that if $p \in S$, then $\refl(p) \in \overline{S}$.

\begin{cor} \label{cor: xyconj orbit and uvconj orbit distinct}
	Every $C_n$-orbit in $\V_{\C}(f,g)$ is disjoint from its image under reflection when $f$ and $g$ satisfy (A1)--(A4).
\end{cor}

\begin{proof}
	Let $\OO$ be a $C_n$-orbit in $\V_{\C}(f,g)$. Suppose $p=[t:x:y] \in \OO\cap \refl(\OO)$. Then 
	$\refl(p) = \rot^{\ell}\cdot p$ for some $\ell \in [n]$, giving that 
	\begin{equation}\label{eq: xyconjref orbit}
	[t:x-iy:x+iy] = \left[t:\omega^{-\ell}(x+iy):\omega^{\ell}(x-iy)\right].
	\end{equation} 
	Then $(x-iy)\cdot \omega^{\ell}(x-iy) = (x+iy)\cdot \omega^{-\ell}(x+iy)$. 
	Taking the modulus of both sides shows that $|x-iy| = |x+iy|$, which 
	contradicts Lemma~\ref{lem: xy ratio real}. Therefore  $\OO\cap\refl(\OO)$ must be empty.
\end{proof}

\begin{remark}
Corollaries~\ref{cor: orbit and xyconj orbit distinct} and \ref{cor: xyconj orbit and uvconj orbit distinct} imply that a $C_n$-orbit $\OO \in \V_\C(f,g)$ is disjoint from both $\conj(\OO)$ and $\refl(\OO)$. However, this tells us nothing about the intersection of orbits $\conj(\OO)$ and $\refl(\OO)$. Their intersection may be nonempty, hence $D_{2n}$-orbits in $\V_\C(f,g)$ do not always have the same cardinality.
\end{remark}

When the matrix $A$ has real entries, both the linear matrix with determinant $F_A$ and its adjugate 
have entries in $\R[t,x+iy, x-iy]$. Therefore to reverse engineer this process and produce 
a matrix in $\CWS_{\R}(n,d)$, we amend the construction to use forms in $\R[t,x+iy, x-iy]$.

\begin{remark}
	Complex conjugation, denoted $\conj$ acts on $\C[t,x,y]$ by conjugating the coefficients of 
	a polynomial in the basis of monomials $t^lx^jy^k$. 
	We claim that the invariant ring of the composition $\refl\circ\conj$ 
	is given by  $\C[t,x, y]^{\langle\refl\circ\conj\rangle}_d = \R[t,x+iy, x-iy]_d$. Indeed,
	any element in $\C[t,x, y]$ is a $\C$-linear combination of forms 
	$t^l (x+iy)^j (x-iy)^k$. Then 
\begin{align*}
		(\refl\circ\conj)\cdot \sum_{l+j+k=d} c_{ljk} t^l (x+iy)^j (x-iy)^k 
		&= \refl\cdot \sum_{l+j+k=d} \overline{c_{ljk}} t^l (x-iy)^j (x+iy)^k\\
			  &= \sum_{l+j+k=d} \overline{c_{ljk}} t^l (x+iy)^j (x-iy)^k, 
	\end{align*}
meaning that the polynomial is invariant if any only if its coefficients $c_{ljk}$ with respect to this basis are real. 
\end{remark}

\begin{lem}\label{lem: dihedral basis}
	If $S \subset \P^2(\C)$ is fixed under $\refl\circ\conj$, i.e.~$\refl(\overline{S}) = S$, 
	then the intersection of the subspace $\es{\ell}_{d-1}$ in \eqref{eq: esl}
	with $ \mathcal{I}(S)_{d-1}$ has a basis in $\R[t,x+iy, x-iy]_{d-1}$.
\end{lem}

\begin{proof}

We will argue that each linear subspace is invariant under $\refl\circ\conj$ separately, hence so is their intersection. 
The subspace $\es{\ell}_{d-1}$ is invariant under $\refl\circ\conj$ since 
in spanned by monomials $t^{d-1-j-k} (x+iy)^j (x-iy)^k$, which are invariant. 
The subspace $\mathcal{I}(S)_{d-1}$ is invariant under $\refl\circ\conj$ because
\[
\mathcal{I}(S)_{d-1} = \mathcal{I}(\refl(\overline{S}))_{d-1} = (\refl\circ\conj)(\mathcal{I}(S)_{d-1}).
\]
Since both $\es{\ell}_{d-1}$ and $\mathcal{I}(S)_{d-1}$ are invariant under $\refl\circ\conj$, so is their intersection. 
It therefore has a basis in $\C[t,x,y]^{\langle \refl\circ\conj\rangle}_{d-1} = \R[t,x+iy, x-iy]_{d-1}$.
\end{proof}

\begin{construction}\label{con: d=qn dihedral}{\rm Let $d=qn$ for some $q \in \Z_+$ and $\Gamma = D_{2n}$.\\
Input: Two plane curves $f$ and $g$ satisfying (A1)--(A4).\\
Output:  a matrix $A\in \CWS_{\R}(n,d)$ such that $f=F_A$.
	\begin{enumerate}
		\item Set $g_{11} = g$.
		\item Split up the distinct $d(d-1)$ points of $\V_{\C}(f,g_{11})$ into two disjoint, conjugate sets  $S \cup \overline{S}$ of $C_n$-orbits such that $ \rot(S)=S$ and $\refl(\overline{S}) = S$.
		\item Extend $g_{11}$ to a linearly independent set $\{g_{11}, g_{12}, \ldots, g_{1d}\} \subset \R[t,x+iy, x-iy]_{d-1}$ vanishing on all points of $S$ with 
		$g_{1j} \in \es{1-j}_{d-1}$ and 
		set $g_{j1} = \overline{g_{1j}}$ for all $j \in [d]$.
		\item For $1<i\leq j$, choose $g_{ij} \in \es{i-j}_{d-1}\cap\R[t,x+iy, x-iy]_{d-1}$ 
		so that $g_{11}g_{ij}-\overline{g_{1i}}g_{1j}$ belongs to $\langle f \rangle$ and $g_{ii} \in \R[t,x,y]$.
		\item For $i < j$, set $g_{ji} = \overline{g_{ij}}$ and define 
		$G = (g_{ij})_{i,j} \in (\R[t,x+iy, x-iy]_{d-1})^{d \times d}$.
		\item Define $M = (1/f^{d-2})\cdot \adj{G}$.
		\item For $\ell \in [d]$, write $\ell-1=an+b$ for some integers $a$ and $b$ with $0 \leq b \leq n-1$. Let $P$ be the permutation matrix that takes $\ell=an+b+1$ to $bq+a+1$. Define $M'=PMP^T$ as a matrix with $q \times q$ blocks $M'=(M'_{kl})_{k,l=1}^n$ 
		\[
		 \left(M'_{kl}\right)_{ij} \in \es{k-l}_1 \text{ for } k,l \in [n] \text{ and }  i,j \in [q].
		\]
		\item For each $k$ compute the Cholesky decomposition of each diagonal block $M'_{kk}$ and write $(M')^{-1}_{kk}= U_{k}U^{T}_{k}$ for some $U_k \in \R^{q \times q}$.
		\item Define $U = \diag(U_1,U_2,\ldots,U_k)$ and output $A=\left(P^{T}U^{T}M'UP\right)(0,1,i)$.
	\end{enumerate}}
\end{construction}

\begin{proof}[Proof of Theorem~\ref{thm: d=qn transverse}(b)]
	 Let $g_{11} = g$. Here we follow Construction~\ref{con: d=qn transverse}, but split the intersection points $\V_{\C}(f,g)$ into $S\cup\overline{S}$ so that $\rot(S) = S$ and $\refl(\overline{S})=S$. 
	 Indeed, by Corollaries~\ref{cor: orbit and xyconj orbit distinct} and \ref{cor: xyconj orbit and uvconj orbit distinct}, 
	 for an orbit $\mathcal{O}$ of a point in $\V_{\C}(f,g)$, we may put $\mathcal{O}$ and  
	 $\refl(\overline{\mathcal{O}})$ in $S$ while taking $\refl(\mathcal{O})$ and $\overline{\mathcal{O}}$ in $\overline{S}$. 
	 By Lemma~\ref{lem: dihedral basis}, we can extend $g_{11}$ to a linearly 
	 independent set $\{g_{11}, \hdots, g_{1d}\}$ so that 
	 $g_{1j} \in \es{1-j}_{d-1} \cap \mathcal{I}(S)_{d-1}\cap\R[t,x+iy, x-iy]$. 
	 Now let $g_{j1} = \overline{g_{1j}}$. 
	             The polynomials $f$, $g_{11}$, $\overline{g_{1i}}g_{1j} \in \R[t,x+iy, x-iy]$, 
            so by Lemma~\ref{lem: maxnoether}, 
            we are also able to find $g_{ij}$ such that $g_{ij} \in \es{i-j}_{d-1} \cap \R[t,x+iy, x-iy]_d$ 
            for $1 < i < j$. Moreover, $g_{ii} \in \R[t,x,y]$ since 
            $f, g_{11}, g_{1i}\overline{g_{1i}} \in \R[t,x,y]$. 
            Let $g_{ji} = \overline{g_{ij}}$ for $i < j$ and define $G = (g_{ij})_{i,j}$. 
            Notice that $G \in \R[t,x+iy, x-iy]_{d-1}^{d \times d}$. 
            We then complete the construction as in the proof of Theorem~\ref{thm: d=qn transverse}. 
            The matrix $M = (1/f^{d-2}) \cdot \adj{G}$ satisfies 
            $M \in \R[t,x+iy, x-iy]_1^{d \times d}$ so $M(1,0,0) \in \R^{d \times d}$. 
            
	 Next we will show by permuting rows and columns of $M$ we may get the identity matrix as the coefficient of $t$ in our representation. Consider $M$ as a matrix of $n \times n$ blocks. Each block is a cyclic weighted shift matrix and there are $q^2$ blocks in total. For $\ell \in [d]$, write $\ell-1=an+b$ for some integers $a$ and $b$ with $0 \leq b \leq n-1$. Let $P$ be the permutation matrix that takes $\ell=an+b+1$ to $bq+a+1$. Define $M'=PMP^T$ as a matrix with $q \times q$ blocks $M'=(M'_{kl})_{k,l=1}^n$ with
		\[
		 \left(M'_{kl}\right)_{ij} \in \es{k-l}_1 \text{ for } k,l \in [n] \text{ and }  i,j \in [q]
		\] and $M'(1,0,0)$ is a real block diagonal matrix. By Theorem 3.3 of \cite{PV}, we know $M(1,0,0)$ is definite, thus $M'(1,0,0)$ is definite. For each $k \in [n]$ write $M'_{kk}(1,0,0)^{-1} =U^{T}_k U_k$ for some $U_k \in \R^{d \times d}$. Define $U = \diag(U_1,U_2,\ldots,U_n)$. Then $M''=U M' U^{T}$ is a representation of $f$ since $M''(1,0,0) = I_d$ and  $f = (1/\lambda) \cdot \det\left(M''\right) \text{ for } \lambda = \det(U)\cdot\det(U^{T})$. Lastly, apply the inverse permutation so $f = (1/\lambda)\cdot\det(P^T M'' P)$. Evaluating $(P^T M'' P)(0,1,i)$ gives a cyclic weighted shift matrix of order $n$ and it is real because $P^TM'' P \in \R[t,x+iy, x-iy]_1^{d \times d}$.
\end{proof}

\begin{example}[$d=6$, $n=3$, $q=2$]	
For  $n=3$ and $d=6$, we see that $\V_{\C}(f,g)$ consists of $d(d-1)  = 30$ points, 
which split into $10$ orbits, each of size $3$. 
Each orbit $\mathcal{O}$ is either fixed by $\refl\circ \conj$, in which case $\refl(\mathcal{O}) = \overline{\mathcal{O}}$, 
or not, in which case $\refl(\mathcal{O}) \neq \overline{\mathcal{O}}$. 
Indeed, since there are 10 orbits total, we see that there must be at least one orbit with 
$\refl(\mathcal{O}) = \overline{\mathcal{O}}$.
Regardless, we can split up the 10 orbits under $C_n$ into two conjugate sets of five. 
The union of each collection is a set $S$ of size 15 satisfying $\refl(\overline{S})=S$. 
 The counts and constructions then continue as in Example~\ref{ex:Counts63}, where 
 the assumption $\refl(\overline{S})=S$ lets us take the forms $g_{1j}$ in $\R[t,x+iy,x-iy]$. 
See \cite[Example~3.2.6]{thesis} for a detailed example of this construction. 
\end{example}

\section{The Degenerate Case}
\label{sec: topology}

 \begin{thm}\label{thm:d=qn}
	Let $d=qn$ for some $q \in \Z_{+}$ and suppose $f \in \HH_d^{\Gamma}$.
	\begin{itemize}
		\item[(a)] If $\Gamma = C_n$, then there exists  $A\in \CWS_{\C}(n,d)$ so that $f=F_A$.
		\item[(b)] If $\Gamma = D_{2n}$, then there exists  $A\in \CWS_{\R}(n,d)$ so that $f=F_A$.
	\end{itemize}
\end{thm}

Here we deal with assumptions (A1)--(A4) posed in Section~\ref{sec: d=qn smooth}. 
To start, we show that the algebraic assumptions hold generically. 

\begin{prop}\label{prop:smooth}
For $d = qn$ and $\Gamma = C_n$ or $D_{2n}$, a generic invariant form $f\in \C[t,x,y]_d^{\Gamma}$ defines a smooth plane curve $\V_{\C}(f) \subset \P^2(\C)$.
\end{prop}
\begin{proof}
Consider the subvariety $\mathcal{X} $ of $\P(\C[t,x,y]_d^{\Gamma}) \times \P^2(\C)$ given by 
\[
\mathcal{X} \ = \ \left \{ (f,p)\in \P(\C[t,x,y]_d^{\Gamma}) \times \P^2(\C) \ : \  \nabla f(p) = (0,0,0)\right \}
\]
By the Projective Elimination Theorem (e.g. \cite[Theorem~10.6]{hassett}), its image under the 
projection $\pi_1(f,p) = f$ is a subvariety of  $\P(\C[t,x,y]_d^{\Gamma})$.  
Therefore the image is either the whole space, meaning that either 
every polynomial in $\C[t,x,y]_d^{\Gamma}$ defines a singular curve, or 
belongs to a proper subvariety of  $\C[t,x,y]_d^{\Gamma}$, meaning that a 
generic polynomial  in $\C[t,x,y]_d^{\Gamma}$ defines a smooth curve. To finish the proof, we note that 
$t^d + (x+i y)^d+ (x-i y)^d$ belongs to $\C[t,x,y]_d^{\Gamma}$ and defines a smooth plane curve. 
\end{proof}

\begin{prop}\label{prop:transverse}
For $d = qn$ and $\Gamma = C_n$ or $D_{2n}$ and any $e\in \Z_+$, 
the plane curves defined by generic invariant forms $f, g\in \C[t,x,y]^{\Gamma}$ with $\deg(f) =d$ and $\deg(g) = e$  intersect transversely. 
\end{prop}

\begin{proof}
First, we argue that it suffices to produce one example of a pair of forms
$f,g$ in $\C[t,x,y]^{\Gamma}$ with $\deg(f) =d$,  $\deg(g)=e$ whose plane curves 
intersect transversely. This is because the intersecting transversely is a Zariski-open condition on $f,g$. 
More precisely, consider the subvariety 
$\mathcal{Y} \subset \P( \C[t,x,y]^{\Gamma}_d) \times \P( \C[t,x,y]^{\Gamma}_{e})  \times \P^2(\C)$ defined by 
\[
\mathcal{Y} =  \biggl\{  (f,g,p):   f(p) = 0, 
g(p) = 0, {\rm rank} \begin{pmatrix} \nabla f(p) \\ \nabla g(p) \end{pmatrix} \leq 1\biggl\}.
\]
Again, by the Projective Elimination Theorem \cite[Theorem~10.6]{hassett}, the image of $\mathcal{Y}$ under the projection $\pi(f,g,p) = (f,g)$ is a Zariski-closed set.  By construction, it is the set of pairs $(f,g)$ for which 
the intersection $\V_{\C}(f)\cap \V_{\C}(g)$ is non-transverse. We need to show that this does not occur for all pairs.

First we consider the special case $d=n$ and $e=1,2$. Note that since $f$ is invariant under the action of $C_n$, it has the form
\[
f(t,x,y) = a(x+i y)^n + b(x-iy)^n +\sum_{j=0}^{\lfloor n/2\rfloor}c_j t^{n-j} (x^2+y^2)^{j}
\]
where $a,b,c_0, \hdots, c_{\lfloor n/2\rfloor}\in \C$ and $a=b$ if $\Gamma=D_{2n}$. 
Note that when $a,b$ are non-zero, the intersection of $\V_{\C}(f)$ with $\V_{\C}(t)$ is transverse for non-zero $a,b$. 
Also if $a,b$ are nonzero, then  $\V_{\C}(f)$ and $\V_{\C}(x^2+y^2)$ have no common points with $t=0$. 
Then by Bertini's theorem, for generic $\lambda, \mu\in \C$, the intersection 
of $\V_{\C}(f)$ and $\V_{\C}(\lambda(x^2+y^2) + \mu t^2)$ is transverse \cite{BertiniBook}.  

Now we construct the desired pair $f, g\in \C[t,x,y]^{\Gamma}$ with $\deg(f) = d$ and $\deg(g)=e$. 
Let $f$ be the product of $q$ generic forms in $\C[t,x,y]_n^{\Gamma}$ of degree $n$, and 
let $g$ be the product of $\lfloor \frac{e}{2}\rfloor$ generic quadratic forms $ \C[t,x,y]_2^{\Gamma}$
and $t^\delta$ where $\delta = 2(\frac{e}{2}  - \lfloor \frac{e}{2} \rfloor)$. 
Then by the argument above, $\V_{\C}(f)$ and $\V_{\C}(g)$ intersect transversely. 
\end{proof} 

\begin{prop}\label{prop:infty}
Let $d = qn$. For 
$\Gamma = C_n$ or $D_{2n}$ and generic invariant forms $f,g$ in $ \C[t,x,y]^{\Gamma}$ with 
$\deg(f) =d$ and $\deg(g) = d-1$, the number of intersection points on the line $t=0$ is given by  
\[
|\V_{\C}(f,g,t)| = \begin{cases} 
0 & \text{ if $n$ is odd,}\\
d & \text{ if $n$ is even.}  
\end{cases}
\]
\end{prop}
\begin{proof}
We first prove something slightly different. 
Let $e\in \Z_+$ be an integer satisfying  $e\in 2\N + n \N$ where $q\cdot e$ is even. 
We claim that generic invariant forms $f,g\in \C[t,x,y]^{\Gamma}$ with 
$\deg(f) =d$ and $\deg(g) = e$ satisfy $\V_{\C}(f,g,t)=\emptyset$. 

By the Projective Elimination Theorem \cite[Theorem~10.6]{hassett}, 
the set of $ (f,g) \in \C[t,x,y]_d^{\Gamma} \times \C[t,x,y]_e^{\Gamma} $
for which $\V_{\C}(f,g,t)$ is non-empty is Zariski closed.
Therefore it suffices to show that it is not the whole space. 

Let $(a,b)\in \N^2$ so that $2a+nb=e$.  
Note that if $e$ is even, then we may take $b$ to be even. 
To see this, note that $e = 2a +nb$ implies that at least one of $b$ and $n$ is even.  If $n=2k$ is even and $b$ is odd, then $b\geq 1$ and 
we may replace the pair $(a,b)$ with $(a+k, b-1)$. 

For an integer $m\in \Z_+$, let $\chi(m)$ be $0$ if $m$ is even and $1$ if $m$ is odd. Then consider polynomials
\[
f = (u^n+v^n)^{\chi(q)}\prod_{j=1}^{\lfloor q/2\rfloor }(u^n+r_jv^n)(r_ju^n + v^n) 
\text{ and }
g = (uv)^a (u^n+v^n)^{\chi(b)}\prod_{k=1}^{\lfloor b/2\rfloor }(u^n+s_kv^n)(s_ku^n + v^n). 
\]
where $r_1, \hdots, r_{\lfloor q/2\rfloor}, s_1, \hdots, s_{\lfloor b/2\rfloor}\in \C\backslash\{0,1\}$  are all distinct and $u = x+iy$, $v=x-iy$. 
We claim that both $f,g$ are invariant under the dihedral group and have no common roots with $t=0$, so long as $\chi(q)\cdot \chi(b)=0$. 
Let $\omega = e^{2\pi i /n}$.  
For invariance, note that both $f$, $g$ are invariant under the change of coordinates $(t,u,v)\mapsto (t,\overline{\omega} u, \omega v)$,  
which is the action of $\rot$ in coordinates $(t,u,v)$, 
as well as the map $(t,u,v)\mapsto (t,v,u)$, which is the action of $\refl$. 

The zeros of $f$ with $t=0$ consist of the points $[t:u:v]=[0:1  : \lambda \omega]$ where $\omega$ is an $n$th root of unity and 
$\lambda=1, r_k, 1/r_k$ for $k=1, \hdots, \lfloor q/2\rfloor$. Moreover there is only such a root with $\lambda=1$ if $q$ is odd. 
Similarly, the zeros of $g$ with $t=0$ consist of the points $[t:u:v]=[0:1:0], [0:0:1]$ if $a\geq 1$ and $[t:u:v]=[0:1  : \lambda \omega]$ 
where $\lambda=1, s_k, 1/s_k$ for $k=1, \hdots, \lfloor b/2\rfloor$, where $\lambda=1$ gives a root only if $b$ is odd. 
Therefore so long as at least one of $q$ or $b$ is even, $\V_{\C}(f,g,t)$ is empty. 

Now suppose that $e=d-1= qn-1$ and $n$ is odd. Then $qn$ has the same parity as $q$, which is different than the parity of $e$. 
Furthermore, $e = (n-1) +(q-1)n$. Since $n-1$ is even, this belongs to $2\N+n\N$.  
The argument from above then shows that $\V_{\C}(f,g,t)=\emptyset$. 

If $n$ is even, then so is $d$, meaning that $d-1$ is odd.  
By Lemma~\ref{lem: factor of t}, every polynomial $g~\in~\C[t,x,y]^{\Gamma}$ of degree $d-1$ has a factor of $t$, meaning that 
it can be written as $g~=~t\cdot~h$ where $h\in \C[t,x,y]^{\Gamma}_{d-2}$. 
Taking $e=d-2$ above shows that $\V_\C(f,h,t)=\emptyset$. Therefore $\V_\C(f,g,t)~=~\V_\C(f,t)$. Since 
$f$ has degree $d$, this consists of $d$ points generically, as is achieved by the explicit example $f$ above. 
\end{proof}

Having dealt with the algebraic conditions of non-singularity, now we address the semi-algebraic conditions of hyperbolicity and interlacing.

\begin{thm} \label{thm: d=qn limit}
For $d = qn$ and $\Gamma = C_n$ or $D_{2n}$, 
every polynomial in $\mathcal{H}_d^{\Gamma} $ 
is a limit of polynomials $f \in (\mathcal{H}^{\circ})_d^{\Gamma}$ for which  
there exists $ g\in \mathcal{H}_{d-1}^{\Gamma}$ such that 
\begin{itemize}
\item[(i)] $\V_\C(f)$ is smooth, 
\item[(ii)] $g$ interlaces $f$, 
\item[(iii)] $\V_{\C}(f)\cap \V_{\C}(g)$  is transverse, and 
\item[(iv)]
$|\V_{\C}(f,g,t)| =\begin{cases} 0 & \text{if $n$ is odd}, \\ d & \text{if $n$ is even.} \end{cases} $
\end{itemize}
\end{thm}

\begin{proof}
For any strictly hyperbolic $f\in \R[t,x,y]_d$ the set of polynomials $g\in \R[t,x,y]_{d-1}$ that interlace $f$ 
with respect to $(1,0,0)$ is a full-dimensional convex cone, 
whose interior consists of those $g$ 
which strictly interlace $f$. See e.g. \cite[Section 6]{KPVinterlacers}. 
Then by Theorem~\ref{thm:InvHyp}, the set  
\[
\mathcal{I} = 
\bigl\{(f,g)\in (\mathcal{H}^{\circ}_{d})^{\Gamma}\times \R[t,x,y]_{d-1}^{\Gamma}: g \text{ strictly interlaces $f$ with respect to } (1,0,0)\bigl\}
\] 
is an open, full-dimensional subset of the affine subspace 
$\{(f,g): {\rm coeff}(f,t^d)=1\}$ in $\R[t,x,y]_d^{\Gamma}\times \R[t,x,y]_{d-1}^{\Gamma}$.  
Moreover, the image of $\mathcal{I}$ under the projection $\pi(f,g)= f$ is all of $ (\mathcal{H}^{\circ}_{d})^{\Gamma}$. 
By Propositions~\ref{prop:smooth}, \ref{prop:transverse}, and \ref{prop:infty},
\[
\mathcal{U} =
\bigl\{(f,g)\in \R[t,x,y]_d^{\Gamma}\times \R[t,x,y]_{d-1}^{\Gamma} : \text{conditions {\rm (i)},{\rm (iii)}, {\rm (iv)} are satisfied}\bigl\}
\]
is open and dense in the Euclidean topology on $\R[t,x,y]_d^{\Gamma}\times \R[t,x,y]_{d-1}^{\Gamma}$. 
Furthermore, we note that $\mathcal{U}$ is invariant under diagonal scaling 
$(f,g)\mapsto (\lambda f, g)$ where $\lambda \in \R^*$. 
An element $(f,g)$ can be rescaled to have ${\rm coeff}(f,t^d)=1$ 
if and only if the coefficient ${\rm coeff}(f,t^d)$ is nonzero, showing that 
$\mathcal{U}$ is also dense in the subspace given by ${\rm coeff}(f,t^d)=1$. 
It follows that $\mathcal{I} \cap \mathcal{U}$ is dense in $\mathcal{I}$. 
Since the projection $\pi(\mathcal{I})$ equals $(\mathcal{H}^{\circ}_{d})^{\Gamma}$, 
this gives that the projection of $\mathcal{I} \cap \mathcal{U}$ is dense in $(\mathcal{H}^{\circ}_{d})^{\Gamma}$. 
Then, by Theorem~\ref{thm:InvHyp}, we see that 
\[
\overline{\pi(\mathcal{I} \cap \mathcal{U})} \ = \ \overline{\pi(\mathcal{I})} \ = \ \overline{(\mathcal{H}^{\circ}_{d})^{\Gamma}} 
 \ = \ \mathcal{H}_d^{\Gamma}.
\]
Therefore every polynomial in $\mathcal{H}_d^{\Gamma}$
belongs to the closure of the set of polynomials $f$ for which there exists $g\in \R[t,x,y]_{d-1}^{\Gamma}$ 
with $(f,g)\in \mathcal{I} \cap \mathcal{U}$. 
\end{proof}

\begin{proof}[Proof of Theorem~\ref{thm:d=qn}]
	         Let $f \in \HH_d^{\Gamma}$. By Theorem~\ref{thm: d=qn limit}, $f$ is the limit of some sequence $(f_{\varepsilon})_{\varepsilon}$ in $(\HH^\circ)_d^{\Gamma}$ satisfying (A1)--(A4). By Theorem~\ref{thm: d=qn transverse}, 
	         for each $\varepsilon$, there exists some matrix $A_\varepsilon$ in $\CWS_{\F}(n,d)$ such that 
	         $f_{\varepsilon} = F_{A_\varepsilon}$, where $\F = \C$ for $\Gamma = C_n$ and $\F = \R$ for $\Gamma = D_{2n}$. Now $f_{\varepsilon}(t,-1,0)$ and $f_{\varepsilon}(t,0,-i)$ are the characteristic polynomials of $\Re(A_\varepsilon) = (A_\varepsilon+A_\varepsilon^{*})/2$ and $\Im(A_\varepsilon) = (A_\varepsilon-A_\varepsilon^{*})/2i$. These must converge to the roots of $f(t,-1,0)$ or $f(t,0,-i)$ respectively. Therefore, the eigenvalues of $\Re(A_\varepsilon)$ and $\Im(A_\varepsilon)$ are bounded, which bounds the sequences $(\Re(A_\varepsilon))_\varepsilon$ and $(\Im(A_\varepsilon))_\varepsilon$. Then
         \[
          (\Re(A_\varepsilon))_\varepsilon + i(\Im(A_\varepsilon))_\varepsilon= (\Re(A_\varepsilon) + i\Im(A_\varepsilon))_\varepsilon= (A_\varepsilon)_\varepsilon
          \] which is also bounded. Passing to a convergent subsequence gives that $\lim \limits_{\varepsilon \to 0}(A_\varepsilon)_\varepsilon = A$ and 
          \begin{align*}
          f&=
          \det \left(
          \lim \limits_{\varepsilon \to 0}
          \left( tI_d + \frac{x+iy}{2} A_{\varepsilon}^* + \frac{x-iy}{2}A_{\varepsilon} \right) \right)
         = \det \left(tI_d + \frac{x+iy}{2}A^* + \frac{x-iy}{2}A \right) = F_A.
          \end{align*} 
\end{proof}


\section{Results on the Classical and $k$-Higher Rank Numerical Range}
\label{sec: num range}

The authors of \cite{Chien2013, Chien2013B, Tsai} were particularly interested in the relationship between the numerical range and the curve dual to its boundary generating curve. Using this relationship and Theorem~\ref{thm:d=qn}, we characterize matrices whose numerical range is invariant under rotation. 

In this section, we describe the interaction between invariance of the numerical range, its boundary generating curve, and the dual variety. We also discuss applications to a generalization of the numerical range. 
In the special case $d=n$, these results appear in \cite{LP}.

	Invariance of $F_A$ under rotation implies the invariance of $\W(A)$ under multiplication by $\omega$, as discussed in 
	Proposition~\ref{prop:CWS=>invariant}. 
	However, the converse does not hold. That is, there are examples for which $\W(A)$ is invariant under
	multiplication by $\omega$, but $F_A$ is not invariant under the action of $C_n$. See Example~\ref{ex: fB not invariant}.
	As discussed below, the invariance of $\W(A)$ only implies that the invariance of some factor of $F_A$, namely 
	the product of irreducible factors whose dual varieties contribute to the boundary of $\W(A)$. 
The invariance of the boundary of $\W(A)$ still gives us information about the dual curve $\V_\C(F_A)$.

\begin{figure}[h]
\centering
	\includegraphics[scale=.4]{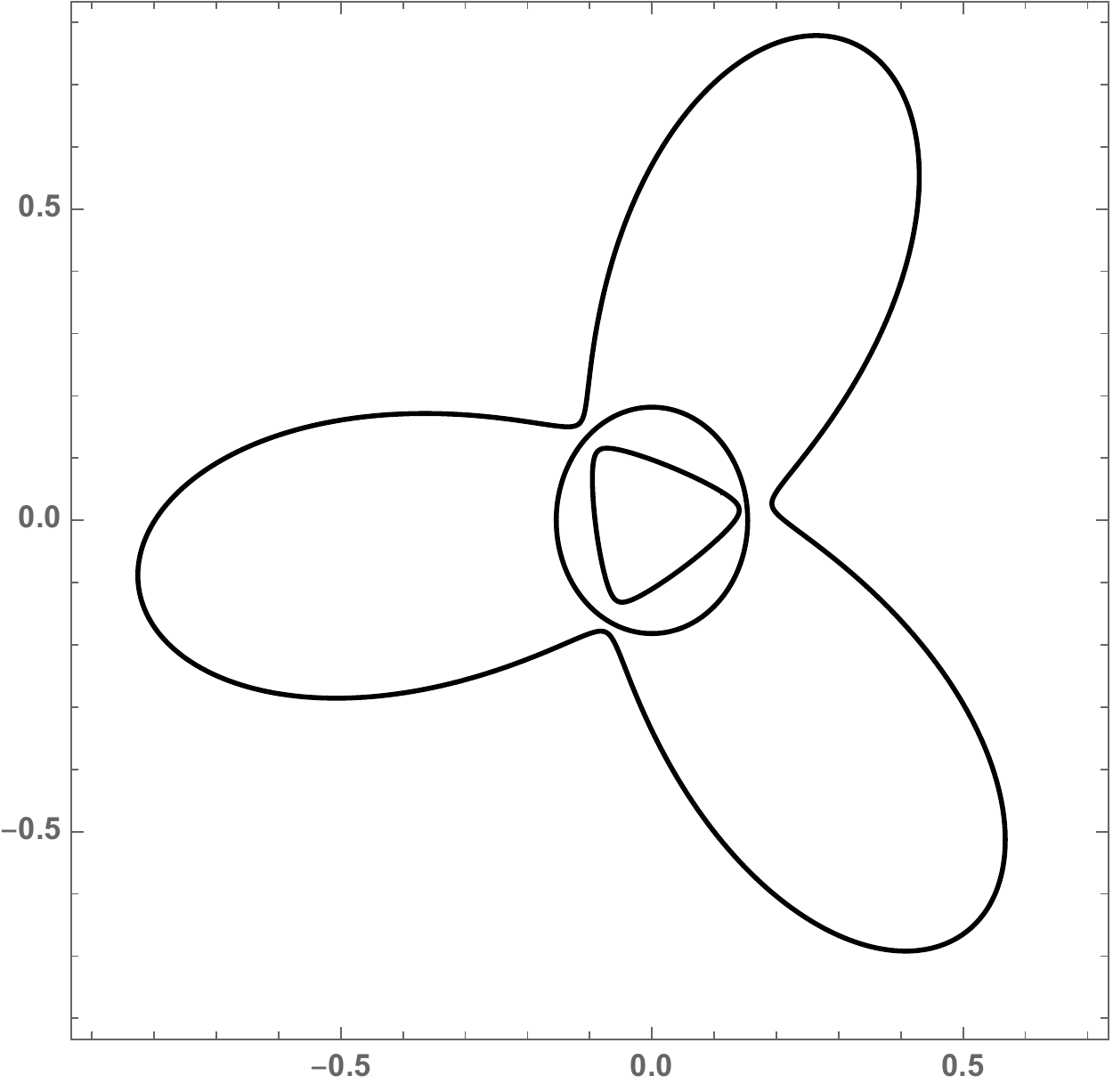}\hspace{5mm}
	\includegraphics[scale=.4]{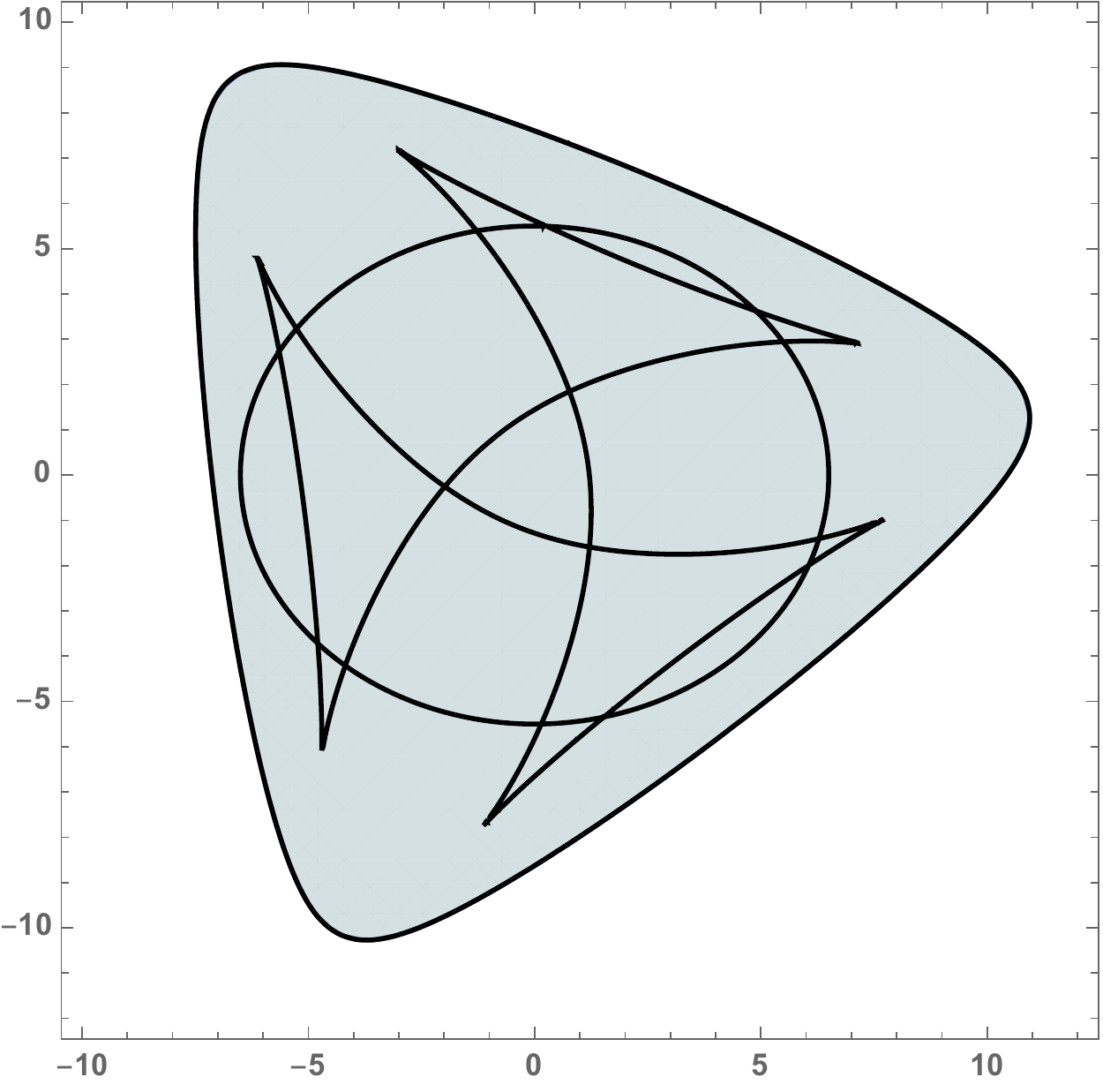}
	\caption{The hypersurface $\V_\R(F_B)$ in the plane $t=1$ for $B \in \C^{6 \times 6}$ from Example~\ref{ex: fB not invariant} (right) and $\W(B)$ (left). Although the plane curve and its dual are not invariant under rotation, the numerical range $\W(B)$ is.}
	\label{fig: fB not invariant}
\end{figure}

       \begin{thm} \label{thm: num range}
Let $B \in \C^{d \times d}$. If  $\mathcal{W}(B)$ is invariant under multiplication by 
$\omega = e^{\frac{2\pi i}{n}}$, then there exists 
$A \in \CWS_{\C}(n, n \lceil d/n \rceil)$ such that $\mathcal{W}(B) = \mathcal{W}(A)$.
If in addition, $\mathcal{W}(B)$ is invariant under conjugation, then $A$ can be taken to have real entries (i.e. $A\in \CWS_{\R}(n, n \lceil d/n \rceil)$.
\end{thm}

\begin{proof}
Kippenhahn's Theorem \cite[Theorem~10]{Kippenhahn} states that 
$\mathcal{W}(B)$  equals the convex hull of 
$\{x+iy : [1:x:y]\in \V_{\R}(F_B)^*\}$.  See also \cite{Chein2010, PSW19}. 
Recall that every compact, convex set is the convex hull of its extreme points. 
Let $E$ denote the extreme points of $\mathcal{W}(B)$ and let  
$Y$ denote the Zarski-closure of  $\{[1:a:b] : a+ib\in E\}$ in $\P^{2}(\R)$. 
By extremality of $E$, we see that $\{[1:a:b] : a+ib\in E\}$ is contained in $\V_{\R}(F_B)^*$ 
and so $Y\subseteq \V_{\R}(F_B)^*$. 
In particular, $Y$ is an algebraic variety of dimension $\leq 1$ and so is a union of points and irreducible curves. 
Moreover, $Y$ contains no lines, since the intersection of $E$ with any line consists of at most two points. 
Therefore all irreducible components of the dual variety $Y^*$ have dimension one. 
Since  $Y\subseteq \V_{\R}(F_B)^*$, we see that  $Y^*\subseteq \V_{\R}(F_B)$. 

Let $f$ denote the minimal polynomial in $\R[t,x,y]_e$ vanishing on $Y^*$. 
Note that since $f$ must be a factor of $F_B$, 
 $e\leq d$, $f$ is hyperbolic with respect to $(1,0,0)$, 
  and $f(1,0,0)\neq 0$, so 
we can take ${\rm coeff}(f,t^e) = 1$. 

Since $\mathcal{W}(B)$ is invariant under multiplication by $\omega$, so is $E$. 
It follows that $Y$ and hence $Y^*$ are invariant under the action of $C_n$. 
Therefore $f\in \R[t,x,y]^{C_n}_e$.  Similarly, if in addition 
$\mathcal{W}(B)$ is invariant under conjugation, then so is $E$. 
The curves $Y$ and $Y^*$ are then invariant under $D_{2n}$ and so $f\in \R[t,x,y]_e^{D_{2n}}$. 

Let $\delta = n\lceil d/n \rceil - e\geq 0$ and consider 
$t^{\delta}f\in \R[t,x,y]_{n\lceil d/n \rceil}^{C_n}$.  
By Theorem~\ref{thm:d=qn}, there exists $A\in \CWS_{\C}(n, n\lceil d/n \rceil)$ such that $F_A=f$.
Then $ \V_{\R}(F_A)^*=V_{\R}(f)^* $ and so $\mathcal{W}(A) = \mathcal{W}(B)$.
Moreover, if $f$ is invariant under $D_{2n}$, then we can take $A$ to be real. 
\end{proof}

\begin{example}\label{ex: fB not invariant}
	Take $B =
{\footnotesize \begin{pmatrix}
 0 \!\! & \!\! 0 \!\! & \!\! 0 \!\! & \!\! 1 \!\! & \!\! 0 \!\! & \!\! 0 \\
 0 \!\! & \!\! 0 \!\! & \!\!\!\! -7-6 i \!\! & \!\! 0 \!\! & \!\! -8-4 i \!\!\!\! & \!\! 0 \\
 0 \!\! & \!\! 0 \!\! & \!\! 0 \!\! & \!\! 0 \!\! & \!\! 0 \!\! & \!\!\!\! -1-2 i \\
 -12 \!\! & \!\! 0 \!\! & \!\! 0 \!\! & \!\! 0 \!\! & \!\! 0 \!\! & \!\! 0 \\
 0 \!\! & \!\! 0 \!\! & \!\! 0 \!\! & \!\! 0 \!\! & \!\! 0 \!\! & \!\!\!\! -10-2 i \\
 0 \!\! & \!\!\!\! -5-10 i \!\!\!\! & \!\! 0 \!\! & \!\! 0 \!\! & \!\! 0 \!\! & \!\! 0 \\
\end{pmatrix}}$. 
Even though $B\not\in \CWS_{\C}(3,6)$, its numerical range 
 $\W(B)$ is invariant under rotation by the angle $2\pi/3$.
 Both $\V_{\R}(F_B)$ and  $\W(B)$ are shown in Figure~\ref{fig: fB not invariant}. 
  For brevity we use $u$, $v$ to denote the linear forms $u = x+iy$ and $v = x-iy$. 
Then $F_B$ factors as $f_1f_2$ where
\begin{align*}
f_1 &= (1/8)\left(8 t^4-798 t^2 u v+1050t (u^3+v^3)+425 i t (u^3-v^3)+3860 (uv)^2\right),\\
f_2 &= (1/4)\left(4 t^2+12 u^2-145 u v+12 v^2\right),
\end{align*}
and $\V_{\C}(f_1)^*$ contains the boundary of $\W(B)$. Notice that $F_B$ is not invariant under rotation by $2\pi/3$, but the quartic factor $f_1$ is. By Theorem~\ref{thm:d=qn}, we can find a matrix $A \in \CWS_{\C}(3,6)$ such that $t^2f_1 = F_A$ and $\W(A) = \W(B)$. One such matrix is given by 
\begin{equation}\label{eq:AwithZeros}
A = {\small \begin{pmatrix}
 0 & -2+i & 0 & 0 & 0 & 0 \\
 0 & 0 & -10+5 i & 0 & 0 & 0 \\
 -6+7 i & 0 & 0 & -4+8 i & 0 & 0 \\
 0 & -2+10 i & 0 & 0 & 0 & 0 \\
 0 & 0 & 0 & 0 & 0 & 0 \\
 0 & 0 & 0 & 0 & 0 & 0 \\
\end{pmatrix}}.
\end{equation}
\end{example}

Theorem~\ref{thm: num range} shows that any invariant numerical range is the numerical range of a block cyclic weighted shift matrix, 
of possibly larger size. One possible strengthening of this is to restrict the the size of this structured matrix. 

\begin{conjecture}
	If $B \in \C^{d \times d}$ and $\W(B)$ is invariant multiplication by 
	$e^{2\pi i /n}$, then there exists $A\in \CWS_{\C}(n,d)$ with $\W(A) = \W(B)$. Moreover, if  $\W(B)$ is also invariant under conjugation, the entries of $A$ can be taken to be real. 
\end{conjecture}

Theorem~\ref{thm:d=qn} also has implications for the following generalization of the numerical range. 
\begin{defn} 
For $k \in [d]$ the \hi{$k$-higher rank numerical range} of $A \in \C^{d \times d}$ is
\begin{equation*}
	\W_k(A):= \left\{\lambda\in \C \text{ such that } PAP = \lambda P \text{ for some rank-$k$ projection }P \right\}.
\end{equation*} 
\end{defn}

This set is compact and invariant under unitary transformation. Building off of the work of Choi et.~al.~\cite{choi}, Woerdeman \cite{woerdeman} showed that $\W_k(A)$ is convex. 
The classical numerical range is defined by $k=1$. Like before, there is a relationship between the geometry of $\W_k(A)$ and the hyperbolic plane curve $F_A$. Chien and Nakazato \cite{Chien2011} describe how to compute $\W_k(A)$ using $F_A$ and the boundary generating curve. They also give conditions for which the $k$-higher rank numerical range is not uniquely determined by the numerical range when $k > 1$.
\begin{remark}
If $A$ is a complex matrix for which $F_A$ is irreducible in $\C[t,x,y]$, then $F_A$ is uniquely 
determined by its numerical range. That is, if $A$ and $B$ are complex matrices for which 
the polynomials $F_A$ and $F_B$ are irreducible, then $\W(A) = \W(B)$ if and only if $F_A = F_B$.  By the results of Gau and Wu \cite{gau2013}, it follows that $\W(A) = \W(B)$ if and only if $\W_k(A) = \W_k(B)$ for all $1 \leq k \leq \floor{d/2}+1$. 
To see this, note that the $\W(A)$ is uniquely determined by its extreme points $E$. 
As in the proof of Theorem~\ref{thm: num range}, if $Y$ is the Zariski closure of 
points $\{[1:a:b]: a+ib\in E\}$, then the minimal polynomial $f$ vanishing on the dual 
variety $Y^*$ is a factor of $F_A$.  If $F_A$ is irreducible, then this gives $f = F_A$, 
which is uniquely determined by $E$ and thus $\W(A)$.
 \end{remark}
Examples from \cite{Chien2015, patton2013, patton2011} show there exist matrices $A$ for which $F_A \in \C[t,x,y]^{C_n}_d$, but $A$ is not unitarily equivalent to any cyclic weighted shift matrix (with positive weights). Theorem~\ref{thm:d=qn} proves there must exist some matrix 
in $\CWS_{\C}(n,n\lceil d/n \rceil)$ with the same $k$-higher rank numerical range of $A \in \C^{d \times d}$, even if the two matrices are not unitarily equivalent.

\begin{cor} \label{cor: cws num range}
    If $ F_B \in \C[t,x,y]^{C_n}_d$ for some $B \in \C^{d \times d}$, then there exists $A \in \CWS_{\C}(n, n \lceil d/n \rceil)$ 
    with $\mathcal{W}_k(A) = \mathcal{W}_k(B)$ for all $1 \leq k \leq \floor{d/2}+1$.
\end{cor}

    \begin{proof}
    Let $m = n\lceil d/n \rceil - d$. 
        By Theorem~\ref{thm:d=qn}, there exists cyclic weighted shift matrix $A \in \CWS_{\C}(n, n\lceil d/n \rceil)$ so that $t^{m}F_B=F_A$. 
 Then \cite[Theorem~1]{gau2013} implies that $\mathcal{W}_k(A) = \mathcal{W}_k(B)$ for all $1 \leq k \leq \floor{d/2}+1$. 
    \end{proof}

    An extension of Theorem~\ref{thm:d=qn} to arbitrary $d$ would yield the following. 
 
\begin{conjecture} \label{cor: cws num range}
    If $ F_B \in \C[t,x,y]^{C_n}_d$ for some $B \in \C^{d \times d}$, then there exists $A \in \CWS_{\C}(n, d)$ 
    with $\mathcal{W}_k(A) = \mathcal{W}_k(B)$ for all $1 \leq k \leq \floor{d/2}+1$.
\end{conjecture}

\section{Open Questions and Further Directions}
\label{sec:questions}

\subsection{Generalizing to Any Degree}\label{subsec:singularities}
One could hope to generalize Construction~\ref{con: d=qn transverse} for a hyperbolic plane curve of any degree. The main obstruction here is with assumption (A1), specifically the requirement that $\V_\C(f)$ is smooth. For curves with $d~\mod~n~\geq~3$, it seems there are always multiple singularities at infinity meaning most of these curves do not satisfy (A1). More specifically, there are complex singularities at the points $[t:x:y]=[0:1:\pm i]$. We conjecture they each have multiplicity $\binom{d \mod n}{2}$.

To see this recall that monomials in $t$, $x+iy$ and $x-iy$ form a basis for 
$\C[t,x,y]_d^{C_n}$, namely 
\begin{equation}\label{eq:monInvariant}
\C[t,x,y]_d^{C_n} = {\rm span}_{\C}\left\{t^{d-j-k}(x+iy)^j(x-iy)^k : j \equiv k \mod n, \text{ and } j+k\leq d \right\}.
\end{equation}
The exponent vectors $(j,k)$ are pictured in Figure~\ref{fig: supports}.

\begin{figure}[h]
\centering
\includegraphics[scale=.4]{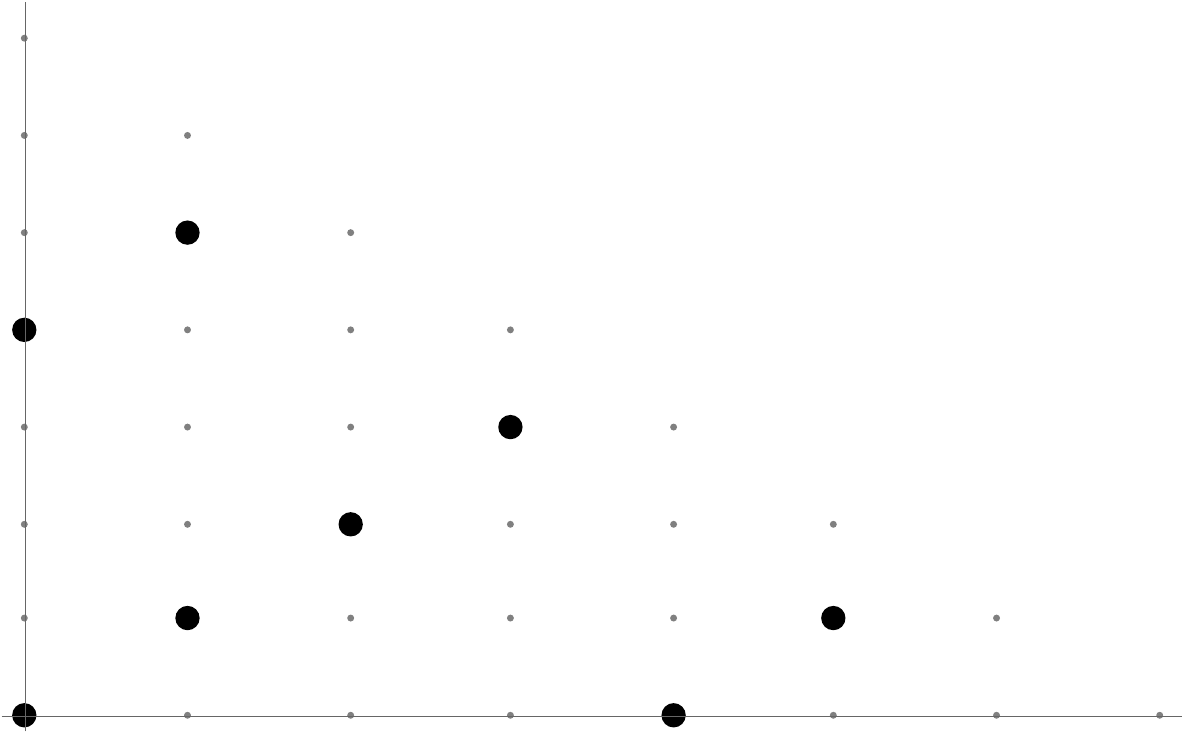}\hspace{.1in}
\includegraphics[scale=.4]{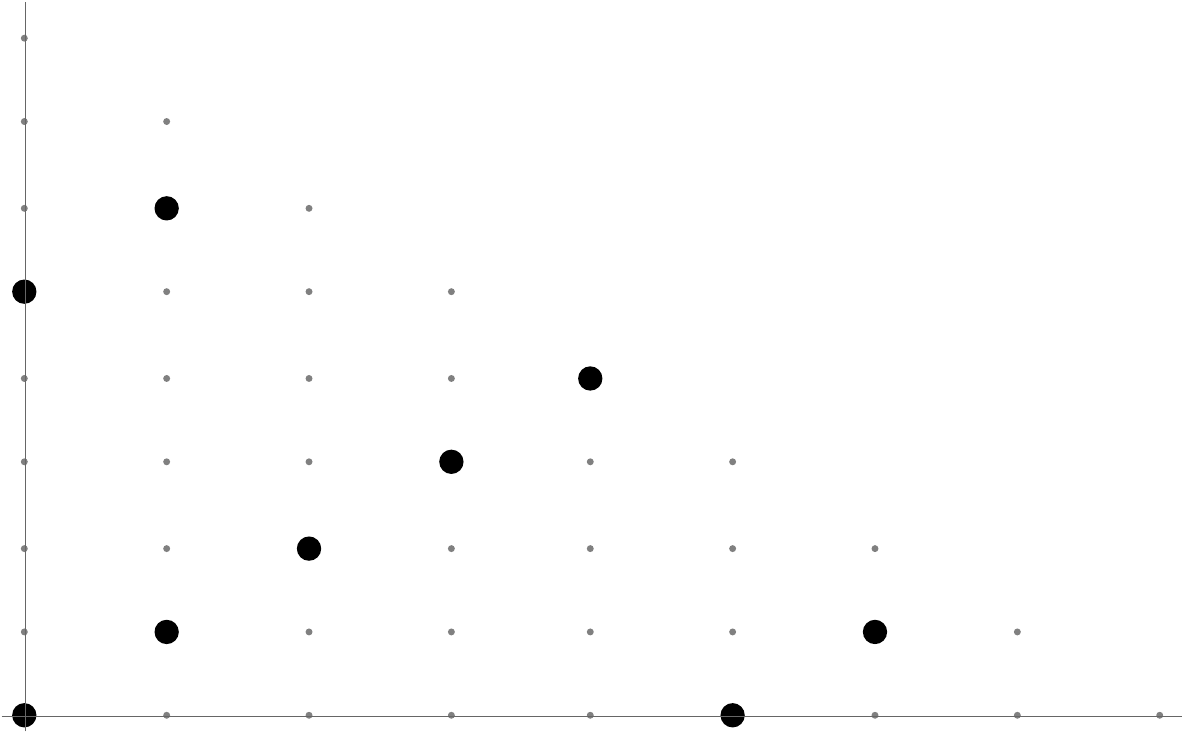}\hspace{.1in}
\includegraphics[scale=.4]{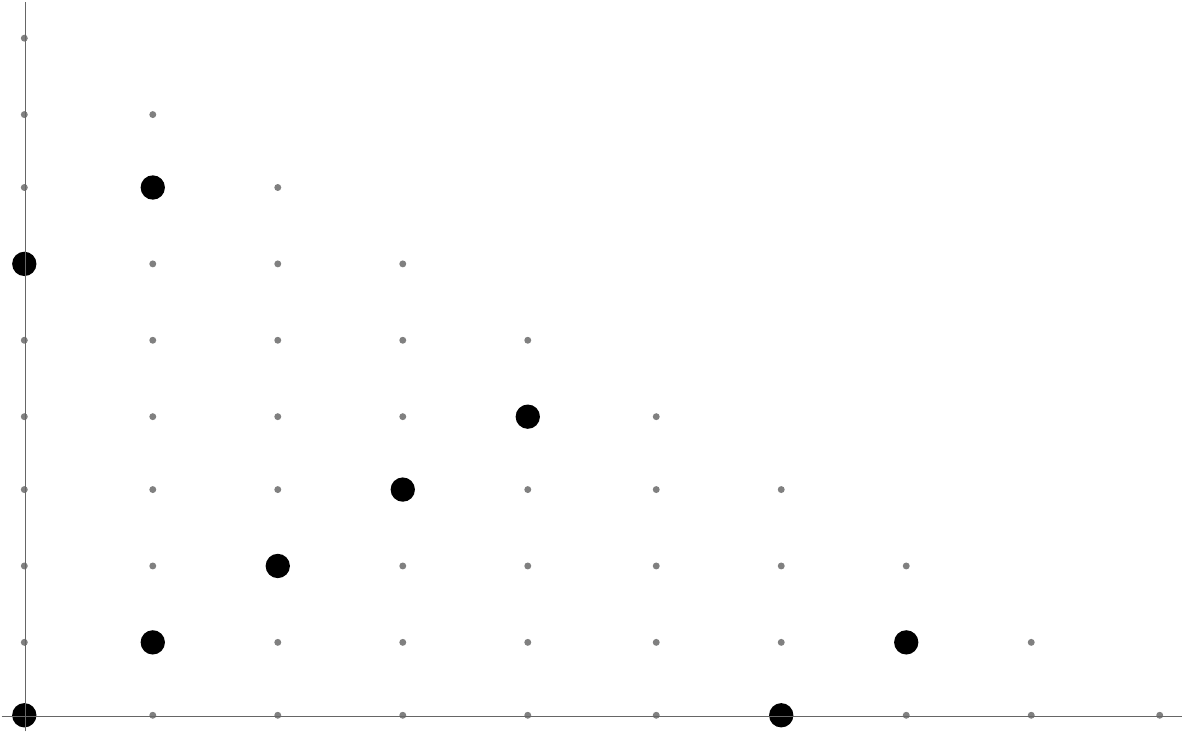}
\caption{
The set of $(j,k)$ for which $t^{d-j-k}(x+iy)^j(x-iy)^k\in\C[t,x,y]_d^{C_n}$
for $(n,d) = (4,7)$, $(5,8)$, $(6,9)$ (left to right).}\label{fig: supports}
\end{figure}

\begin{example}[$d=7, n=4$] 
Consider $f \in \C[t,x,y]_7^{C_4}$. Then $f$ is a sum of terms of the form 
$t^{7-j-k}(x+iy)^j(x-iy)^k$ where $j\leq 5$ and $k\leq 5$, shown on the left in Figure~\ref{fig: supports}. 
This confirms that $[0:1:\pm i]$ are singular points of $\V_{\C}(f)$.
\end{example}

Another way to try to construct a determinantal representation is to use Theorem~\ref{thm:d=qn} and hope to further specialize its structure.

\begin{question}\label{q: d=qn+m det rep degeneration}
	Let $d=qn+m$ for some $q >1$ and $m \in [n-1]$ and suppose $f~\in~\HH_d^{\Gamma}$. 
	Can we always write $t^{n-m}f = F_A$ for some matrix $A \in \CWS_{\C}(n,(q+1)n)$ 
	of the form $A = \begin{pmatrix}A' & \textbf{0}\\  \textbf{0} &  \textbf{0} \end{pmatrix}$, where 
	$A'\in  \CWS_{\C}(n,d)$?
	\end{question}

\begin{example} Take $f = f_1$ from Example~\ref{ex: fB not invariant}. 
This is quartic and invariant under action of the group $C_3$. 
Then $t^2f=F_A$ for the matrix $A\in \CWS_{\C}(3,6)$ shown in equation \eqref{eq:AwithZeros}. 
Since the last two rows and columns of $A$ are zero, 
$f$ has a determinantal representation $f=F_{A'}$ where $A'$ is the 
leading $4 \times 4$ minor of $A$. 
\end{example}

\subsection{Higher Dimensions} 
One can also consider invariant hyperbolic polynomials and determinantal representations in 
more than three variables.  
Suppose that $\Gamma\subset \GL(\R^n)$ is a finite group that fixes a point ${\bf e}\in \R^n$ and 
let  $\mathcal{H}_d^{\Gamma}$ denote the set of polynomials in $f\in \R[x_1, \hdots, x_n]_d$ invariant under $\Gamma$, hyperbolic with respect to ${\bf e}$, and with $f({\bf e}) = 1$, 
as in Section~\ref{sec:Nuij}. 
 
\begin{question} 
Is the analogue of Theorem~\ref{thm:InvHyp} true in higher dimensions? That is, are both $\mathcal{H}_d^{\Gamma}$ and its interior contractible? 
\end{question}

As shown in Theorem~\ref{thm:d=qn}, hyperbolic polynomials invariant under $C_n$ and $D_{2n}$
have determinantal representations that \emph{certify} their invariance. 
More generally, let  $\rho:\Gamma \rightarrow \GL(\C^d)$ be a representation of the group $\Gamma$. 
This defines an action of $\Gamma$ on the set of $d\times d$ Hermitian matrices by conjugation, 
$\gamma \cdot A =  \rho(\gamma) A  \rho(\gamma)^*$. 
We say that a $d\times d$ linear matrix $\mathcal{A}({\bf x})  = \sum_i x_i A_i$ is invariant with respect to 
$\Gamma$ and $\rho$ 
if for every $\gamma\in \Gamma$, 
\begin{equation}\label{eq:invariantDetRep}
\mathcal{A}(\gamma \cdot {\bf x}) =  \rho(\gamma) \mathcal{A}({\bf x})  \rho(\gamma)^*.
\end{equation}
The determinant $f = \det(\mathcal{A}({\bf x}))$ is then invariant under the action of $\Gamma$. 
Indeed, since $\Gamma$ is finite, the determinant of $\rho(\gamma)$ is a root of unity 
and so the determinants of $\rho(\gamma)$ and  $\rho(\gamma)^*$ multiply to $1$. 
This shows that the determinants of $\mathcal{A}(\gamma \cdot {\bf x})$ and $\mathcal{A}({\bf x})$ are equal for all $\gamma \in \Gamma$. 

\begin{example}\label{ex:MoreVar}
The elementary symmetric function $e_{n-1}(x_1, \hdots, x_n) = \sum_{k=1}^n\prod_{j\neq k}x_j$ 
is hyperbolic with respect to the vector ${\bf e} = (1,\hdots, 1)$ and invariant under the 
natural action of the symmetric group $S_n$. 
Sanyal \cite{Sanyal} shows that 
the form $e_{n-1}({\bf x}) $ has a determinantal representation $\mathcal{A}({\bf x}) = {\rm diag}(x_1, \hdots, x_{n-1}) + x_nJ$, 
where $J$ is the all-ones matrix of size $n-1$.  This representation is invariant with respect to $S_n$ 
and the representation $\rho: S_n \to \GL(\C^{n-1})$ obtained by restricting $S_n$ to the 
hyperplane of points with coordinate sum one. 
Specifically, we take the representation $\rho(\pi) = (v_{\pi(1)}, \hdots, v_{\pi(n-1)})$ where for $j=1, \hdots, n$, $v_j$ is the $j$th unit coordinate vector in $\R^{n-1}$ and $v_n$ is the constant vector $-1$. 
Specializing $e_7(x_1, \hdots, x_8)$ and its determinantal representation to the eight linear forms  
$x_j = t\pm x \pm y\pm z$ gives a surface in $\P^3(\R)$ that is hyperbolic and invariant under the octahedral group,
shown in Figure~\ref{fig:icermLogo}.  
\end{example}

\begin{figure}
\includegraphics[height=2in]{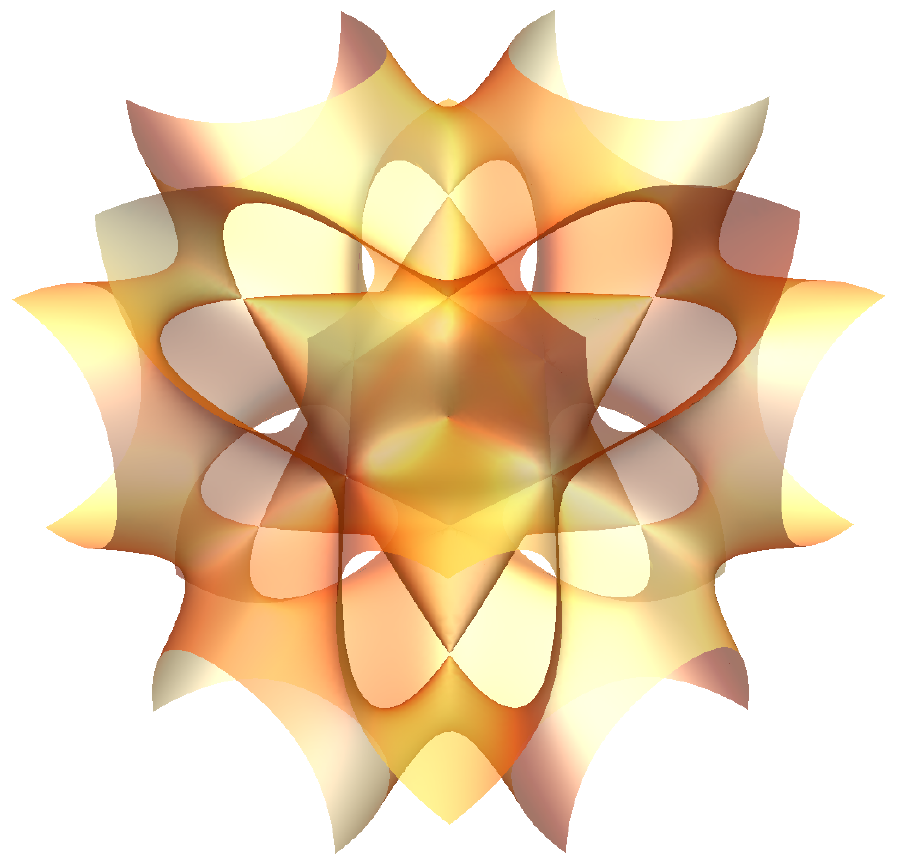}
\caption{A hyperbolic surface in $\P^3(\R)$ invariant under the octahedral group 
with an invariant definite determinantal representation given in Example~\ref{ex:MoreVar}.}\label{fig:icermLogo}
\end{figure}

For $n>3$, most forms in $\R[x_1, \hdots, x_n]_d$ do not have $d\times d$ determinantal representations, 
so a verbatim generalization of Theorem~\ref{thm:d=qn} is false. 
However, there are two other natural ways of generalizing to more variables. 
One is to restrict to polynomials that are already determinantal: 
\begin{question} 
For every finite group $\Gamma\subset \GL(\R^n)$, is there a representation $\rho:\Gamma \rightarrow  \GL(\C^d)$ 
so that every determinantal hyperbolic polynomial $f\in \mathcal{H}_d^{\Gamma}$ has a 
definite determinantal representation $f = \det(\mathcal{A}({\bf x}))$ that is invariant with respect to $\Gamma$ and 
$\rho$, as in \eqref{eq:invariantDetRep}? 
\end{question}

A more ambitious goal would be to show that every invariant hyperbolic polynomial has 
determinantal representation certifying its hyperbolicity and invariance. 
For this, we use the terminology of \emph{hyperbolicity cones}. 
If a polynomial $f\in \R[x_1, \hdots, x_n]$ is hyperbolic with respect to ${\bf e}\in \R^n$, 
its hyperbolicity cone, $C(f,{\bf e})$ is defined to be the connected component of $\R^n\backslash \V_{\R}(f)$ 
containing ${\bf e}$. 
The Generalized Lax Conjecture states that for every hyperbolic polynomial $f$, there is some multiple 
$f\cdot g$ with a definite determinantal representation so that the hyperbolicity cones 
of $f$ and $f\cdot g$ agree. This is still open. 
The discussion above suggests the following invariant version:

\begin{question}[Invariant Generalized Lax Conjecture]
Is every invariant hyperbolic polynomial a factor of an invariant determinant? 
That is, for $f\in \mathcal{H}_d^{\Gamma}$, does there exist $e\in \N$, $g\in \mathcal{H}_e^{\Gamma}$, 
and a representation $\rho:\Gamma \rightarrow  \GL(\C^{d+e})$ so that the product 
$f\cdot g$ has an invariant, definite determinantal representation $f\cdot g = \det(\mathcal{A}({\bf x}))$ with $C(f,{\bf e})=C(f\cdot g,{\bf e})$? 
\end{question}


\end{document}